\documentclass[oneside,reqno]{amsart}
\usepackage{algpseudocode}
\usepackage{graphicx}
\usepackage{caption}
\usepackage{amsmath,amssymb}
\usepackage{empheq}
\usepackage{color, tikz, comment}
\usepackage[utf8]{inputenc}
\usepackage{mathtools}
\usepackage{float}
\usepackage{subfigure}
\usepackage[ruled,linesnumbered,lined,boxed,commentsnumbered]{algorithm2e}

\def\one{\mbox{1\hspace{-4.25pt}\fontsize{12}{14.4}\selectfont\textrm{1}}}
\newtheorem{thm}{Theorem}[section]
\newtheorem{cor}[thm]{Corollary}
\newtheorem{lem}[thm]{Lemma}
\newtheorem{prop}[thm]{Proposition}
\newtheorem{assump}[thm]{Assumption}
\theoremstyle{definition}
\newtheorem{defn}[thm]{Definition}
\theoremstyle{remark}
\newtheorem{rem}[thm]{Remark}
\numberwithin{equation}{section}
\newcommand{\defeq}{\vcentcolon=}

\renewcommand{\parallel}{\mathrel{/\mkern-5mu/}}
\makeatletter
\newcommand{\notparallel}{%
  \mathrel{\mathpalette\not@parallel\relax}%
}
\newcommand{\not@parallel}[2]{%
  \ooalign{\reflectbox{$\m@th#1\smallsetminus$}\cr\hfil$\m@th#1\parallel$\cr}%
}
\makeatother

\theoremstyle{plain}

\newcommand{\norm}[1]{\left\Vert#1\right\Vert}

\newcommand{\loc}{\mathrm{loc}}

\begin{document}

\title[Bayesian inversion for a class of tumor growth models]{A unified Bayesian inversion approach for a class of tumor growth models with different pressure laws}

\author{Yu Feng, Liu Liu, Zhennan Zhou}%

\address{Yu Feng: Beijing International Center for Mathematical Research, Peking University, No. 5 Yiheyuan Road Haidian District, Beijing, P.R.China 100871}
\email{fengyu@bicmr.pku.edu.cn}
\address{Liu Liu: Department of Mathematics, The Chinese University of Hong Kong, Lady Shaw Building, Ma Liu Shui, Hong Kong, China}
\email{lliu@math.cuhk.edu.hk}
\address{Zhennan Zhou: Beijing International Center for Mathematical Research, Peking University, No. 5 Yiheyuan Road Haidian District, Beijing, P.R.China 100871}
\email{zhennan@bicmr.pku.edu.cn}
\date{\today}

\subjclass[2020]{35R30, 62F15, 65M32, 92-10}%

\keywords{Bayesian inversion, tumor growth models, asymptotic-preserving}%
\thanks{}

\maketitle

\begin{abstract}
In this paper, we use the Bayesian inversion approach to study the data assimilation problem for a family of tumor growth models described by porous-medium type equations. The models contain uncertain parameters and are indexed by a physical parameter $m$, which characterizes the constitutive relation between density and pressure. Based on these models, we employ the Bayesian inversion framework to infer parametric and nonparametric unknowns that affect tumor growth from noisy observations of tumor cell density. We establish the well-posedness and the stability theories for the Bayesian inversion problem and further prove the convergence of the posterior distribution in the so-called incompressible limit, $m \rightarrow \infty$. Since the posterior distribution across the index regime $m\in[2,\infty)$ can thus be treated in a unified manner, such theoretical results also guide the design of the numerical inference for the unknown. We propose a generic computational framework for such inverse problems, which consists of a typical sampling algorithm and an asymptotic preserving solver for the forward problem. With extensive numerical tests, we demonstrate that the proposed method achieves satisfactory accuracy in the Bayesian inference of the tumor growth models, which is uniform with respect to the constitutive relation.   
\end{abstract}


\section{Introduction}
In recent years, mathematical modeling has become an increasingly important tool in tumor research. By using mathematical models to simulate tumor growth and evolution, one can better understand the underlying mechanisms that drive tumor progression. However, most existing work on mathematical models in tumor research is limited to formulation and analysis, which means they are designed to predict how a tumor will develop given certain initial conditions and parameters. And it needs to be emphasized that due to the limitations in understanding the tumor growth mechanism, various models exist in the current literature, such as stochastic models based on reaction-diffusion equations \cite{greenspan1976growth}, phase field models based on Cahn-Hilliard equations \cite{garckehilliard}, and mechanical models based on porous media equations \cite{perthame2014hele}. We suggest the following textbooks \cite{cristini2010multiscale,cristini2017introduction} and review articles \cite{araujo2004history,byrne2006modelling,roose2007mathematical,lowengrub2009nonlinear} as references for interested readers.

As tumor growth is a rather complex biological process, it develops in distinguishable phases and is affected by various factors. Many mathematicians are devoted to incorporate these elements in modeling and analyze their individual and synergistic effects, such as nutrient concentration \cite{feng2023tumor,jacobs2022tumor}, degree of vascularization \cite{cristini2003nonlinear,pham2018nonlinear}, cell reproduction and apoptosis \cite{friedman2001symmetry, friedman2008stability}, chemotaxis \cite{lu2020complex,lu2022nonlinear}. However, the development of the model library also raises an alarming issue, the model identification and the parameter calibrations in the equations are becoming significantly more challenging as well. 

The presence of unknown parameters and the difficulty of validating models against experimental data are major obstacles in the practical application of these tumor models. Therefore, studying the inverse problem in tumor growth has both theoretical and practical values. For example, by conducting model selection and parameter inferences, researchers can gain insights into the underlying mechanisms driving tumor growth and progression \cite{falco2023quantifying}. Also, The inverse problem can be used to optimize treatment strategies for individual patients by predicting the efficacy of different treatments \cite{kostelich2011accurate, lipkova2019personalized}. 

The study on the inverse problem for tumor growth has a shorter history compared to the forward modeling but has received significant attention in recent years. In the context of tumor growth modeling, the inverse problem aims to estimate the unknown parameters in the model (e.g., proliferation rates, diffusion coefficients, etc.) that govern the growth of tumors via the observed data such as tumor images or size measurements \cite{falco2023quantifying,kahle2020parameter,kahle2019bayesian,lipkova2019personalized,subramanian2020did}. Moreover, various methodologies have also been developed for concerning the inverse problem in tumor growth models, such as Tikhonov regularization method \cite{kahle2020parameter,liu2014patient}, Bayesian inference \cite{falco2023quantifying,lipkova2019personalized}, Machine learning algorithms \cite{selvanambi2020lung,zhang2019spatio} and so on.

In particular, among the methodologies above, Bayesian inference has emerged as a promising approach for solving the inverse problem in tumor growth modeling \cite{falco2023quantifying,kahle2019bayesian,kostelich2011accurate,lipkova2019personalized}. This approach involves combining prior knowledge about the unknown model parameters with likelihood functions that capture the probability of observing the available data.  Bayesian methods have been used to estimate parameters in various tumor growth models, such as reaction-diffusion model \cite{ lipkova2019personalized}, phase field model \cite{kahle2019bayesian}, and mechanical model (degenerate diffusion model) \cite{falco2023quantifying}. Additionally, Bayesian approaches can be combined with Uncertainty Quantification (UQ) methods to generate probabilistic predictions of tumor growth dynamics, providing insight into the uncertainty associated with the estimated model parameters and guiding us in assessing the reliability and robustness of the estimated parameters and their predictions.

Despite the progress made in inverse problems and UQ studies for tumor growth, many challenges remain. In particular, due to the diversity and hierarchy in the model library, it becomes inefficient to design tailored treatments for specific models.

In this paper, we consider the inverse problem of a family of mechanical models for tumor growth described by porous-medium type equations. The tumor cell density evolves as follows
\begin{equation} \label{eqn:pme for short}
    \frac{\partial}{\partial t}\rho - \nabla\cdot\left(\rho \nabla p \right)=g(x,t,\rho), \quad p=\frac{m}{m-1}\rho^{m-1}, \quad m\ge 2.
\end{equation}
Here, $\rho$ denotes the cell density, $p$ denotes the pressure and $g$ is the growth factor. For simplicity, we take $g=h(x) \rho$, where $h$ is the growth rate function manifesting the local condition of the growing environment.
 We can index these models according to the physical parameter $m$, which specifies the constitutive relation between density and pressure. 
 
 Such models share the same physical laws but obey different constitutive relations, a phenomenon that is reminiscent of kinetic models containing different collision kernels or fluid mechanical models with different pressure relations \cite{friedlander2002handbook, villani2002review}. It is worth mentioning that the physical parameter is also similar to the scaling parameter $\varepsilon$ in multiscale models \cite{nolen2012multiscale, weinan2011principles}, but they also differ significantly, since as $m$ varies, the nonlinearity structure changes as well, which cannot be recovered by rescaling.

Without loss of generality, we consider two types of unknowns in the inverse problem: the non-parametric and the parametric ones. The former refers to unknown functions without additional assumptions on their functional forms, such as the growth rate function $h$, and the latter refers to finite-dimensional parameters associated with unknowns in some prescribed forms, such as shape parameters specifying the initial profile. 

In this work, we study the Bayesian inversion problem for model \eqref{eqn:pme for short} indexed by $m  \in I= [2,\infty)$, and aim to provide a unified computation framework for such inverse problems. To be more precise, the numerical method is supposed to not only produce stable and reliable parameter inference for each model with fixed $m$, but also we expect that the numerical results should exhibit uniform accuracy across the index regime $m\in I$. In particular, it is necessary to rule out the possibility that  the numerical performance degenerates as $m\rightarrow \infty$.

From the Bayesian point of view, we seek a probabilistic solution to the inverse problem in the form of a posterior distribution $\mu_m^y$,  where $y$ denotes the observed data (which will be omitted in this section) and $m$ is the physical index. However, since the posterior distribution is often formidably high dimensional (or even possibly infinite-dimensional), sampling tools are applied to obtain a statistical presentation of the distributions. In this sense, proposing a unified computational framework for these inverse problems boils down to designing a numerical method that can efficiently sample the collection of posterior distributions $\{ \mu_m \}_{m \in I}$. 

Our analysis of the Bayesian problems investigates the properties of the posterior distributions and thus provides theoretical foundations and insights for constructing the numerical scheme.  On one hand, we establish the well-posedness theory for the Bayesian inversion problem with a given index $m$; on the other hand, we show that the posterior distributions converge in the limit $m\rightarrow \infty$. These results strongly yield a key observation: the probability measures in the set  $\{ \mu_m \}_{m\in I}$ do not differ much besides being absolute continuous with respect to the prior distribution. 

In light of this, most prevailing numerical sampling strategies, such as Markov Chain Monte Carlo (MCMC) methods, can be adopted here. Notice that  when generating each sample a typical numerical scheme involves computing the likelihood function, which requires efficiently computing the forward problem. Thus a reliable numerical solver for the tumor growth models, which achieves  correct approximations for $m\in I$, is desired. Thanks to the previous works \cite{jianguo2018,jianguo2021}, an asymptotic preserving numerical scheme has been constructed, which can accurately capture the boundary moving speed in the limit $m\rightarrow \infty$. Hence, such numerical schemes can readily be integrated into our numerical method for the inverse problem.

To sum up, the unified computational method for the Bayesian inversion problems to a family of tumor growth models consists of a  plain MCMC method and an asymptotic preserving numerical solver for the forward problem. We highlight that our theoretical analysis only indicates the minimal requirements for treating the collection of posterior distributions, and it is certain that more advanced sampling techniques can be applied to further improve the numerical performance. 

We that compared with other prevailing inverse problem approaches, the Bayesian approach avoids finding the estimator of the inverse or solving the optimization problem  with a regularized functional, thus it offers plenty of flexibility in dealing with different models with the same approach.  In a recent paper \cite{falco2023quantifying}, the authors also adopt the Bayesian inversion method to compare different tumor growth models and confirm that the pme-based models \eqref{eqn:pme for short} are more reasonable in the presence of tissue collision.

This paper is organized as follows: in Section \ref{sec: Preliminary}, we introduce a family of tumor growth models described by the porous medium type equations, and set up the Bayesian inverse problem for these models, and present the unified numerical method. In Section \ref{sec: section3}, we establish the well-posedness and stability theory for the Bayesian inversion problems and characterize the convergence behavior of the posterior distributions in the incompressible limit, which serve as the theoretic foundation for the numerical scheme.  The numerical experiments are presented in Section \ref{sec: Numerical Experiments} to verify our results in theoretical analysis. Lastly, the conclusion and future work is addressed. 

\section{Preliminary}
\label{sec: Preliminary}
In this section, we begin with introducing a family of  tumor growth models indexed by a physical parameter $m$, which are porous medium type equations and possess a Hele-Shaw-type asymptote as the index $m$ tends to infinity. Then, we formulate the inverse problems with respect to the above models and employ a Bayesian framework to quantify parametric and nonparametric  unknowns in the models based on some noisy observation data.  In the last part of this section, we establish the algorithm for the inverse problem, which works for an extensive range of index $m$ and can capture the asymptotic limit of the solutions.

\subsection{A family of deterministic tumor growth model}
\label{sec:A deterministic tumor growth model}
In the first part, we adopt and introduce a family of well-studied mechanical tumor growth models that are porous medium type equations and are indexed by a physical parameter $m$ specifying the constitutive relation between the pressure and the density (see \cite{benilan1996singular}, section 3 ). In each mechanistic model, i.e., fixing a value of the index $m$, we consider the evolution of the tumor cell density over a specified domain. Moreover, as the physical index $m$ tends to infinity, such equations have natural Hele-Shaw type asymptotes. For a complete introduction to the model, we begin with introducing the notation and physical parameters.

Let $\Omega$ be a bounded open set in $\mathbb{R}^2$, and we consider the growth of the tumor in this region. For $T>0$, define $Q_T\defeq\Omega\times(0,T)$, and $\Sigma_T\defeq\partial\Omega\times(0,T)$. Let $\rho(x,t)$ denote the cell population density, with the cells transported by a velocity field $v$ and the cell production governed by the growth function $g(x,t,\rho)$. Then the continuity of mass yields 
\begin{equation}
\label{eqn:cts of mass}
    \frac{\partial}{\partial t}\rho+\nabla\cdot\left(\rho v\right)=g(x,t,\rho).
\end{equation}
We further assume the velocity $v$ is governed by Darcy's law $v=-\nabla p$, where the pressure $p$ further satisfies the power law: $p=\frac{m}{m-1}\rho^{m-1}$, with $m (\geq 2)$ meanwhile acts as the index for the family of problems. Then the continuity of mass equation \eqref{eqn:cts of mass} can be further written into:
\begin{equation}
\label{eqn:pme density eqn}
    \frac{\partial}{\partial t}\rho-\nabla\cdot\left(\rho \nabla p\right)=g(x,t,\rho).
\end{equation}
Moreover, we employ the set
\begin{equation}
\label{eqn:set D}
    D(t)=\left\{\rho(x,t)>0\right\}
\end{equation}
to denote the support of $\rho$. Physically, it presents the tumoral region at time $t$. Then the tumor boundary expands with a finite normal speed $s=-\nabla p\cdot {\bf n}\vert_{\partial D}$, where ${\bf n}$ stands for the outer normal vector on the tumor boundary. 
Observe the fact that the expression of $p$ enables the flux $\nabla\cdot(\rho \nabla p)$ equivalently written as $\Delta \rho^m$. On the other hand, for the boundary condition, we assume $\rho$, so as $p$, vanishes on $\Sigma_T$. Besides, let $f(x)$ be the initial data, and it can generally be an arbitrary function that takes the value in $[0,1]$. However, in practice, we focus on a specific class of initial data, which can simplify the regime. We leave the detailed explanation for later.  

With the above assumptions, for any $m\geq 2$, the evolution of the tumor cell density satisfies the following system:
\begin{subequations}
\label{eqn:Pm}
\begin{empheq}[left={(P_m)\empheqlbrace\,}]{align*}
&\rho_t=\Delta\rho^{m}+g(x,t,\rho)\qquad&\text{on}\quad Q_{T},\\
&\rho= p = 0 \qquad&\text{on}\quad\Sigma_T,\\
&\rho(x,0)=f(x)\qquad&\text{on}\quad \Omega.
\end{empheq}
\end{subequations}
For each fixed $m\geq 2$, the system $(P_m)$ possesses a unique solution (see Theorem \ref{thm:exist_Pm}) under proper assumptions. In this work, we consider the growth function in the following form
\begin{equation}
\label{eqn:nonlocalFKPP g}
g(x,t,\rho)= h(x)\rho, \quad 0<h(x)\in L^{\infty}(\Omega).
\end{equation}
The expression in \eqref{eqn:nonlocalFKPP g} can be understood as the cell production is determined by the cell density and a growth rate function $h(x)$, which reflects the tumor micro-environment that may affect cell growth, such as the distribution of nutrients. Moreover, we consider the development of an early-stage tumor so that the cell apoptosis is neglectable, and $h(x)$ can be reasonably assumed to be a positive function. 

Many research (e.g. \cite{david2021free, dou2022modeling, guillen2022hele, kim2016free, kim2018porous, perthame2014hele}) indicate that the porous medium type functions have a Hele-Shaw type asymptote as the power $m$ tends to infinity. In particular, the solution of $(P_m)$ tends to the solution of (see Theorem \ref{thm:L1 convergence thm} for precise description):
\begin{subequations}
\label{eqn:P_infty}
\begin{empheq}[left={(P_\infty)\empheqlbrace\,}]{align*}
&\rho_t=\Delta p_{\infty}+g(x,t,\rho)\qquad&\text{on}\quad Q_{T},\\
&0\leq\rho\leq1,\quad p_{\infty}\geq 0,\quad(\rho-1)p_{\infty}=0\qquad&\text{on}\quad Q_{T},\\
&p_{\infty}= 0 \qquad&\text{on}\quad\Sigma_T,\\
&\rho(x,0)=f(x)\qquad&\text{on}\quad \Omega,
\end{empheq}
\end{subequations}
if the initial data $f$ is provided to be a characteristic function $f=\chi_{D_0}$, where $D_0$ be a bounded subset of $\Omega$. That means the initial density is saturated in the set $D_0$ and vanishes outside. Actually, $(P_m)$ converges to $(P_\infty)$ for more general initial data (see Theorem \ref{thm:L1 convergence thm}). However, the prescribed ones can simplify the regime and are enough for our purpose. And it is worth mentioning that in the Hele-Shaw model $(P_{\infty})$, if the initial data is in the form of a characteristic function, then the solution remains in the form of characteristic function consistently, i.e., $\rho(x,t)=\chi_{D(t)}$. We refer to these solutions as patch solutions. 
Furthermore, for patch solutions and $g(x,t,\rho)$ given by \eqref{eqn:nonlocalFKPP g} with $h(x)>0$, the limit pressure $p_{\infty}(x,t)$ solves the following elliptic problem in the tumoral region $D(t)$ for each time $t$:
\begin{align}
   -\Delta p_{\infty}&=h(x)\quad\text{in}\quad D(t),\\
   p_{\infty}&\geq 0\quad\text{in}\quad D(t),\\
   p_{\infty}&=0\quad\text{on}\quad \partial D(t).
\end{align}
And the tumor boundary propagates with a finite normal speed $s=-\nabla p_{\infty}\cdot {\bf n}\vert_{\partial D}$.

\subsection{Set up for the inverse problem}
In this section, we set up the inverse problem based on the models established in the previous section. 

For each $m\geq 2$, consider the model $(P_m)$ with $g(x,t,\rho)$  given by \eqref{eqn:nonlocalFKPP g}. And the initial data in the form of $f=f_0^{z}(x)$, where $f_0$ is a given characteristic function $\chi_{D_0}$, with $D_0\subset \mathbb{B}_1(0)$, i.e., a subset of the unit disk centered at the origin. And $z$ can generally be any parameters for the initial data with a prescribed form, such as the center and the scaling (or size). Then the problem $(P_m)$ can be further written as:
\begin{subequations}
\label{eqn:Pm'}
\begin{empheq}[left={(P_m')\empheqlbrace\,}]{align*}
&\rho_t=\Delta\rho^{m}+h(x)\rho\qquad&\quad\text{on}\quad Q_{T},\\
&\rho=p= 0 \qquad&\text{on}\quad\Sigma_T,\\
&\rho(x,0)=f_0^{z}(x)\qquad&\text{on}\quad \Omega.
\end{empheq}
\end{subequations}
Our primary interest is identifying two types of unknowns in the problem $(P_m')$ from some noisy observations that will be specified later. The first unknown type collects the unknowns from the parametric form of the initial data. This type of unknowns constitute a simple finite-dimension vector. While the second kind of unknown is treated in a non-parametric way, such as the growth rate function $h(x)$. For concision, we collect them in a single variable $u$ as
$$u=(z,h(x)).$$
Given  $\hat{u}=(\hat{z},\hat{h}(x))$, $(P_m')$ has a unique solution (see Theorem \ref{thm:exist_Pm}), and we denote it as $\rho^{(m)}:=\rho^{(m)}(\hat{u})$. For the observations, we consider data obtained from snapshots of the tumor at several time instances, which are slightly polluted by noises. We assume that the noises cannot be directly measured but their statistical properties are known. In the work, the noises are modeled as Gaussian random variables which are independent of the unknown parameters. Mathematically, we generate the noisy observation with respect to $\rho^{(m)}$ as follow: 
\begin{enumerate}
    \item Fix a sequence of smooth test function $\{\xi_k\}_{k=1}^{K}$  with $supp(\xi_k)\subseteq\Omega$ for any $1\leq k\leq K$. 
    \item Fix $T>0$, and let $\{t_j\}_{j=1}^{J}$ (with some fixed $J\in\mathbb{N}$) be an increasing sequence in the time interval $[0,T]$.
    \item We model the noisy observations using a set of linear functional $\{l_{j,k}\}_{j=1,k=1}^{j=J,k=K}$ of the solution $\rho^{(m)}$. Specifically, we assume that $l_{j,k}:f\mapsto l_{j,k}(f)\in \mathbb{R}$ is given by
       \begin{equation}
       \label{eqn:linear-functional}
       l_{j,k}(f)=\int_{\Omega}\xi_k(x)f(x,t_j)dx.
       \end{equation}
    Then the noisy observations, denote by $\{y_{j,k}\}_{j=1,k=1}^{j=J,k=K}$, $y_{j,k}\in \mathbb{R}$, are expressed as
\begin{equation}
\label{eqn:single observation}
    y_{j,k}^m=l_{j,k}(\rho^{(m)})+\eta_{j,k},\qquad 1\leq j\leq J,\quad 1\leq k\leq K, 
\end{equation}
where $\eta_{j,k}\sim N(0,\sigma^2_{j,k})$, i.e., the standard normal distribution with mean $0$ and variance $\sigma^2_{j,k}>0$.
\end{enumerate}

For concision, let data space $Y\defeq \mathbb{R}^{J K}$. Define the noise vector $\eta\coloneqq(\eta_{j,k})\in Y$ and the observation vector $y\coloneqq(y_{j,k})\in Y$ with $1\leq j\leq J$ and $1\leq k\leq K$.
Then \eqref{eqn:single observation} can be written in the vector form:
\begin{equation}
\label{eqn:noisy observation for m}
    y=\mathcal{G}^m(\hat{u})+\eta,
\end{equation}
where the forward operator $\mathcal{G}^m(\hat{u})$ is the composition of the solution operator $\mathcal{F}^m:=\hat{u}\mapsto \rho^{(m)}(\hat{u})$ and the observation functionals $\rho^{(m)}\mapsto l_{j,k}(\rho^{(m)})$, with $1\leq j\leq J$ and $1\leq k\leq K$. And the noise vector 
\begin{equation}
\label{eqn:noise}
   \eta\sim N(0,\Gamma), 
\end{equation}
where the covariance matrix $\Gamma$ is a $JK$ by $JK$ diagonal matrix with diagonal elements given by $\sigma^2_{j,k}>0$. 

For the inverse problem, we assume that $m$ can be measured directly from experiment data, and we consider the following inverse problem: given $m$ and the noisy data $y$, we aim to infer the unknown $\hat{u}$ by \eqref{eqn:noisy observation for m} in a probability sense. 

On the other hand, it is worth emphasizing that we aim to solve for a family of inverse problem indexed by $m$, which takes value in a semi-bounded domain $[2, \infty)$. And thus, it is inevitable to discuss the solution behavior as $m$ is approaching infinity. 

As explained previously, the solution to $(P_m')$ converges to the solution of the following one
\begin{subequations}
\label{eqn:P_infty'}
\begin{empheq}[left={(P_\infty')\empheqlbrace\,}]{align*}
&\rho_t=\Delta p_{\infty}+h(x)\rho\qquad&\text{on}\quad Q_{T},\\
&0\leq\rho\leq 1,\quad p_{\infty}\geq 0,\quad(\rho-1)p_{\infty}=0\qquad&\text{on}\quad Q_{T},\\
&p_{\infty}= 0 \qquad&\text{on}\quad\Sigma_T,\\
&\rho(x,0)=f^z_0(x)\qquad&\text{on}\quad \Omega.
\end{empheq}
\end{subequations}
Let $\rho^{(\infty)}(\hat{u})$ be the solution to $(P_{\infty}')$, with $(z,h(x))$ replaced by $(\hat{z},\hat{h}(x))$, then one can define $(\mathcal{F}^\infty, \mathcal{G}^\infty)$ in the same way as $(\mathcal{F}^m, \mathcal{G}^m)$. More precisely, each component of the observation vector $y$ is given by
\begin{equation}
y_{j,k}=l_{j,k}\circ\mathcal{F}^\infty(\hat{u})+\eta_{j,k}:=\mathcal{G}_{j,k}^\infty(\hat{u})+\eta_{j,k}.
\end{equation}
In the forward problem, one has $\rho^{(m)}(\hat{u})\rightarrow\rho^{(\infty)}(\hat{u})$ in proper function space (see Theorem \ref{thm:L1 convergence thm}). For the inverse problem, since we aim to design a numerical method that works for a large range of physical index $m$, we not only require that the approach is uniformly well-posed for $m\in [2,\infty)$, but also we expect the numerical performance does not degenerate as $m$ approaches infinity. 

We employ a Bayesian approach for the inverse problem to identify the unknown factor $\hat{u}$. The Bayesian inversion is a method for solving inverse problems by using Bayes' theorem to update our beliefs about the unknown parameters by leveraging the observed data. We take the identification for problem $(P_m')$ as an example, and the identification for problem $(P_\infty')$ can be done similarly. 

To begin with, we treat the unknown $\hat{u}$ as a random variable. To distinguish with the deterministic $\hat{u}$, we use the notation $u$ for the random variables instead. Recall that $u$ contains two types of components. For the parameter $z$, we assume it generates from a uniform distribution and denote the measure as $\mu_0^z$. While for the random function $h(x)$, we assume it can be presented as:
\begin{equation*}
    h(x)=h_0(x)+\Sigma_{j=1}^{\infty}\gamma_j\zeta_j\phi_j,
\end{equation*}
where $h_0(x)$ is a determined positive $L^{\infty}$ function,  $\gamma=\{\gamma_i\}_{i=1}^{\infty}$ is a deterministic sequence of scalars, $\phi=\{\phi_i\}_{i=1}^{\infty}$ is a set of basis functions for a certain function space, and $\zeta=\{\zeta_i\}_{i=1}^{\infty}$ be an i.i.d. random sequence with $\zeta_i\sim N(0,1)$, thus we have defined the prior distribution for $h$, which we denote by $\mu_0^h$  (see, e.g. \cite{dashti2017bayesian}); Therefore, $u$ has a \emph{priori measure} $\mu_0:=\mu_0^z\times\mu_0^h$, since $z$ and $\zeta$ (so as $h(x)$) are independent. We leave the precise description of $\mu_0$ to Section \ref{sec:Generation and prior measure for the random variables}.

The posterior distribution obtained from Bayesian inversion represents our beliefs about the parameters and their uncertainty after data assimilation. We aim to derive the posterior distribution with respect to the noisy observation data $y$, which we denote as $\mu_m^y$.  The classical theory of Bayes' rule yields the following Radon-Nikodym relation \cite{dashti2017bayesian} with respect to $\mu_m^y$ and $\mu_0$:
\begin{equation}
    \frac{d\mu_m^{y}}{d\mu_0}(u,y)=\frac{1}{Z_{m}(y)}\exp{(-\Phi_m(u,y))},\quad
    Z_{m}(y)=\int_{X}\exp{(-\Phi_m(u,y))}d\mu_0(u),
    \label{RN1}
\end{equation}
where the potential function $\Phi_m(u,y)$ takes the form of:
\begin{equation}
    \Phi_m(u,y)=\frac{1}{2}\vert\Gamma^{-1/2}(\mathcal{G}^m(u)-y)\vert^2-\frac{1}{2}\vert\Gamma^{-1/2}y\vert^2.
    \label{RN2}
\end{equation}
Recall that $\mathcal{G}^m$ is the forward operator as in \eqref{eqn:noisy observation for m}, and $\Gamma$ is the covariance matrix for the observation noise. And we can define $(\Phi_{\infty},Z_{\infty}, \mu_{\infty}^y)$ analogously.

We devote ourselves to the following three main targets in the following:
\begin{enumerate}
    \item Show that the Bayesian inversion problem is well-posed to all $m\geq 2$.
    \item Show that the posterior distribution $\mu_m^y$ converges as $m$ tends to infinity, in the sense of Hellinger distance (see Definition \ref{def:H-distance}).
    \item With the theoretical understanding above, design a numerical method for the inverse problem that works uniformly well for $m\in [2,\infty)$.
\end{enumerate}
We close this section by presenting the numerical method in the next subsection and leaving the first two targets to the latter chapters.

\subsection{Algorithm for the inverse problem}
In this section, we establish a unified computational method for the family of tumor growth models in the Bayesian inversion framework. Recall that in the tumor growth model 
\begin{equation}
\label{model_eqn}
\left\{ 
\begin{array}{ll}
\partial_t \rho = \Delta \rho^m+ h( \mathbf{x})\rho, \qquad \mathbf{x} \in \Omega, \\[4pt]
\rho(\mathbf{x},0) = \rho_0(\mathbf{x}+\mathbf{z}), 
\end{array}
\right.
\end{equation}
we aim to infer the unknown $u^\star=(z,h(x))$ based on the noisy observation data 
\begin{equation*}
  y = \mathcal{G}^m(u^{\star}) + \eta.   
\end{equation*}
Here, the forward operator $\mathcal{G}^m=l\circ\mathcal{F}^m$ is the composition of the solution operator $\mathcal{F}^m$ and the linear functional $l$, defined in \eqref{eqn:linear-functional}-\eqref{eqn:noisy observation for m}, and $\eta\sim N(0,\Gamma)$ as in \eqref{eqn:noise}  denotes the noise.

In the Bayesian inversion, given the prior distribution $\mu_0$ and the noisy observation data $y$, the distribution law of the unknown $u$ is given by the posterior distribution $\mu_m^y$, and by \eqref{RN1} and \eqref{RN2} we know that
\begin{equation}
    \frac{d\mu_m^{y}}{d\mu_0}(u) \propto \exp{\left(-\frac{1}{2}\vert\Gamma^{-1/2}(\mathcal{G}^m(u)-y)\vert^2\right)}.
    \label{RN3}
\end{equation}
Since the normalization constant in \eqref{RN1} is not feasibly computable, and the posterior distribution $\mu_m^y$ or its finite-dimensional approximations are of complicated landscapes, we seek numerical sampling of $\mu_m^y$ based on \eqref{RN3} rather than a direct computation. 

It is worth emphasizing that our primary goal is to construct a numerical method that can efficiently sample $\mu_m^y$ for arbitrary $m \ge 2$, hence two major challenges are in the way. On the one hand, the collection of measures $\mu_m^y$ needs to be investigated such that the criterion for the choice of the sampler can be established. On the other hand, during the sampling process, one needs to repeatedly solve  $\mathcal{G}^m(u)$, which evolves solving the PDE model \eqref{model_eqn}, and thus a numerical solver that works for all $m \ge 2$ is desired.

Our proposed method consists of two main ingredients: a plain MCMC method and an asymptotic-preserving (AP) numerical solver for the forward problem. We elaborate on the designing principle and implementation details in the following. 

In Section \ref{sec: section3}, we will give theoretical proof that the posterior distribution $\mu_m^y$ is well-posed and stable for each $m$ and further show that it converges as $m\to\infty$. This guarantees that the posterior distribution behaves as a Cauchy sequence (refer to Theorem \ref{thm:convergence of posterior}) so that it does not vary dramatically as $m$ increases. Due to the similarity among the posterior distributions with different $m$, a standard sampling method would be sufficient to accomplish the task, hence we choose the plain MCMC approach and briefly review it below. More advanced sampling techniques will be considered in future work.

For simplicity, we employ a typical MCMC algorithm called the Metropolis-Hastings (MH) that constructs a Markov chain by accepting or rejecting samples extracted from a proposed distribution and is widely used in Bayesian inversion \cite{dashti2017bayesian}. 


Now, we are ready to lay out our main algorithm. Here, the covariance $\Gamma^u$ determines the Markov kernel that generates sampling proposals.

\begin{algorithm*}[htb]
\caption{Metropolis-Hastings MCMC}
Assign a proper $u_0$.

Generate the next proposal $ u' \sim \mathcal{N}(u_k, \Gamma^{u})$, where $\Gamma^{u}$ is a given diagonal with each diagonal element $\Gamma^{u}_{ii}>0$.  \\[2pt]

Calculate the acceptance probability $$\displaystyle\alpha(u',u_k) = \min\Bigl\{1, \frac{ L(y|u')\mu_0(u')}{ L(y|u_k) \mu_0(u_k)}\Bigr\}.$$ Here, $\mu_0$ stands for the prior distribution, the likelihood function $$\displaystyle L(y|u) \propto \exp{\left(-\frac{1}{2}\vert\Gamma^{-1/2}(\mathcal{G}^m(u)-y)\vert^2\right)}, $$ and to obtain $\mathcal{G}^m(u)$, one needs to solve the PDE model \eqref{model_eqn} and calculate the observations \eqref{eqn:linear-functional}. \\[2pt]

Update  as $u_{k+1} = u'$ with probability $\alpha(u', u_k)$, otherwise 
set $u_{k+1} = u_k$. 
\end{algorithm*}

We observe  that in the step of calculating the likelihood and the acceptance rate, the quantity $\mathcal{G}^m(u)$ is computed by calling the forward solver. Hence, a unified  Bayesian inversion approach for such a class of tumor growth models is not completed without an efficient and robust forward solver that works for all $m \ge 2$.  

With the constitutive law of $p(\rho) = \frac{m}{m-1}\rho^{m-1}$, both the nonlinearity and the degeneracy in diffusion bring significant challenges in numerical simulations. An asymptotic-preserving (AP) scheme that can accurately capture the boundary moving speed in the limit $m\to\infty$ is necessary. Despite the above challenges, thanks to the previous work \cite{jianguo2018,jianguo2021}, we adopt the AP scheme developed there as our forward solver and briefly summarize it below.

In \cite{jianguo2018}, a numerical scheme based on a novel prediction correction reformulation that can accurately approximate the front propagation has been developed. The authors show that the semi-discrete scheme naturally recovers the free boundary limit equation as $m \to \infty$. With proper spatial discretization, this fully discrete scheme has been shown to improve stability, preserve positivity, and can be implemented without nonlinear solvers. For convenience, we summarize the numerical scheme developed in \cite{jianguo2018,jianguo2021} in the Appendix. By using this AP solver, we can compute the density solution and then obtain $\mathcal{G}(u)$, thereby the likelihood functions in Step 3 of \textbf{Algorithm 1}.

To sum up, we have constructed a generic computational framework for the inverse problem of tumor growth models for all $m \ge 2$, which consists of a plain MCMC method and the AP forward solver originated in \cite{jianguo2018}. In the next, we shall provide both  a theoretical foundation as well as extensive numerical verification for the effectiveness of the proposed method.

\section{Well-posedness, stability, and convergence for the posterior distribution}
\label{sec: section3}
In this section, we establish the well-posedness and stability results for the Bayesian inversion problems of $(P_m')$ and $(P_\infty')$. We emphasize that these results are held uniformly for the physical index $m\in I$. In the last part of this chapter, to further exclude the possibility that the posterior diverges in the incompressible limit, where $m$ tends to infinity, we prove that the posterior distribution indeed converges in the sense of the Helllinger distance.

\subsection{Well-posedness and $L^1$ contraction for the forward problem}
We devote this section to establishing the well-posedness and properties of the forward problems, which also served as the cornerstone for showing the well-posedness, stability, and convergence of the posterior distribution in the inverse problems. 

Consider problem $(P_m)$, and we begin with recalling the results from \cite{benilan1996singular}. Firstly, we make following assumptions for the initial data $f$, and the growth function function $g(x,t,\rho)$.
\begin{assump}
\label{assump: f and g}
Let $f\in L^{\infty}(\Omega)$ with $f\geq 0$, and $g:Q_T\times\mathbb{R}_{+}\rightarrow\mathbb{R}$ satisfies:
\begin{enumerate}
    \item[(i)] $g(x,t,r)$ is continuous in $r\in\mathbb{R}_+$ for a.e. $(x,t)\in Q_T$,
    \item[(ii)] $g(\cdot,r)\in L_{\loc}^1(\bar{\Omega}\times[0,T))$ for any $r\in\mathbb{R}_+$,
    \item[(iii)] $\frac{\partial g}{\partial r}(x,t,\cdot)\leq K(\cdot)$ in $\mathcal{D}'(0,\infty)$ for a.e. $(x,t)\in Q_T$ with $K\in\mathcal{C}(\mathbb{R}_+)$,
    \item[(iv)] $g(\cdot,0)\geq 0$ a.e. on $Q_T$,
    \item[(v)] there exists $M\in W_{\loc}^{1,1}([0,T))$ such that $M'(t)\geq g(x,t,M(t))$ for a.e. $(x,t)\in Q_T$ and $M(0)\geq\norm{f}_{L^{\infty}(\Omega)}$.
\end{enumerate}
\end{assump}
The above assumptions implies
\begin{equation*}
    g(\cdot,\rho)\in L_{\loc}^1\left(\bar{\Omega}\times[0,T)\right)\qquad\text{for any}\quad\rho\in L_{\loc}^{\infty}(\Omega\times[0,T))
\end{equation*}
since
\begin{equation}
    g(\cdot,R)-\Tilde{K}(R)R\leq g(\cdot,r)\leq g(\cdot,0)+\Tilde{K}(R)R\quad\text{for}\quad 0\leq r\leq R,
\end{equation}
where $\Tilde{K}(R)=\max_{[0,R]} K$.

Under above assumptions, one has the well-posedness for $(P_m)$. We give the precise description in the following.

\begin{thm}[Lemma 2,\cite{benilan1996singular}]
\label{thm:exist_Pm}
Under Assumption \ref{assump: f and g}, for any $m\geq 2$ there exists a unique solution of $(P_m)$ in the sense
\begin{subequations}
\begin{empheq}[left={\empheqlbrace\,}]{align*}
&\rho\in L^{\infty}_{\loc}([0,T)\times\Omega)\cap \mathcal{C}([0,T);L^1(\Omega)),\quad\rho\geq 0,\quad \rho(\cdot, 0)=f(\cdot),\\
&\rho^m\in L^2_{\loc}([0,T);H^1(\Omega))\quad\text{and}\quad \frac{\partial\rho}{\partial t}=\Delta \rho^m+g(\cdot,\rho) \quad\text{in}\quad \mathcal{D}'(Q_T).
\end{empheq}
\end{subequations}
Moreover $\rho\leq M$ a.e. on $Q_T$. 
\end{thm}

Besides the well-posedness of the problems $\left\{(P_m)\right\}_{m=2}^{\infty}$, the convergence of $(P_m)$ to $(P_\infty)$ is characterized as following.

\begin{thm}[Theorem 2,\cite{benilan1996singular}]
\label{thm:L1 convergence thm}
Under Assumption \ref{assump: f and g}, for $m\geq 2$, let $\rho^{(m)}$ be the solution of $(P_m)$ given in Theorem \ref{thm:exist_Pm}. Then,
\begin{enumerate}
    \item $\rho^{(m)}\rightarrow \rho^{(\infty)}$ in $\mathcal{C}((0,T);L^1(\Omega))$ as $m\rightarrow\infty$.
    \item Assuming $g(\cdot,1)\leq\Tilde{g}$ in $\mathcal{D}'(Q_T)$ with $\Tilde{g}\in L_{\loc}^2([0,T),H^{-1}(\Omega))$, then there exists a unique $(\rho,p_{\infty})$ solution of $(P_{\infty})$ in the sense
    \begin{subequations}
\begin{empheq}[left={\empheqlbrace\,}]{align*}
&\rho\in \mathcal{C}((0,T);L^1(\Omega)),\quad p_{\infty}\in L_{\loc}^2((0,T),H_0^1(\Omega)),\\
&\rho(\cdot, 0)=\Tilde{f}(\cdot),\quad 0\leq \rho\leq 1,\quad p_{\infty}\geq 0,\quad (\rho-1)p_{\infty}=0,\\
&\frac{\partial\rho}{\partial t}=\Delta p_{\infty} + g(\cdot,\rho)\quad\text{in}\quad \mathcal{D}'(Q_T),
\end{empheq}
\end{subequations}
where $\Tilde{f}=f\chi_{[\Tilde{p}_{\infty}=0]}+\chi_{[\Tilde{p}_{\infty}>0]}$, with $\Tilde{p}_{\infty}$ the unique solution of the 'mesa problem':
\begin{subequations}
\begin{align*}
    &\Tilde{p}_{\infty}\in H_0^1(\Omega),\quad\Delta\Tilde{p}_{\infty}\in L^{\infty}(\Omega),\quad\Tilde{p}_{\infty}\geq 0,\\
    & 0\leq\Delta\Tilde{p}_{\infty}+f\leq 1,\quad \Tilde{p}_{\infty}(\Delta\Tilde{p}_{\infty}+f-1)=0\quad\text{a.e.}\quad\Omega.
\end{align*}
\end{subequations}
And we have $\rho^{(\infty)}=\rho$.
\end{enumerate}
\end{thm}
\begin{rem}
It is easy to check that the assumptions for $g(x,t,r)$ in the Assumption~\ref{assump: f and g} covers not only the form of \eqref{eqn:nonlocalFKPP g} but also the standard FKPP form employed in \cite{falco2023quantifying}. On the other hand, if the initial data $f$ is in the form of a characteristic function, then $f=\Tilde{f}$. And we only consider the initial data in such a form. Thus, $(P_m')$ and $(P_{\infty}')$ as sub-case of $(P_m)$ and $(P_m)$ correspondingly. Theorem \ref{thm:exist_Pm} and Theorem \ref{thm:L1 convergence thm} provide the existence and uniqueness of solution to $(P_m')$ and $(P_{\infty}')$ respectively. And Theorem \ref{thm:L1 convergence thm} also characterize the convergence of $(P_m')$ to $(P_{\infty}')$.
\end{rem}

Next, we introduce the so-called $L^1$-contraction property with respect to $(P_m)$ and $(P_\infty)$. Such property is inherited from the porous medium type equations.  Now, we begin with the case for $(P_m)$.

\begin{thm}[Theorem 1.1,\cite{igbida2021a}]
\label{thm:rho m L1-cotraction}
For each $m\geq 2$, if $\rho_1$ and $\rho_2$ are two solutions of $(P_m)$ associated with $g_1$ and $g_2$ satisfying Assumption \ref{assump: f and g} respectively, then 

\begin{equation*}
    \frac{d}{dt}\norm{\rho_1-\rho_2}_{L^1}\leq\norm{g_1-g_2}_{L^1},\qquad\text{in}\quad\mathcal{D}'(0,T).
\end{equation*}
\end{thm}
On the other hand, the limit problem $(P_{\infty})$ possess similar property.  
\begin{thm}[Theorem 2.1,\cite{igbida2021b}]
\label{thm:rho infty L1-cotraction}
If $(\rho_1,p_1)$ and $(\rho_2,p_2)$ are two solutions of $(P_\infty)$ associated with $g_1$ and $g_2$ satisfying Assumption \ref{assump: f and g} respectively, then
\begin{equation*}
    \frac{d}{dt}\norm{\rho_1-\rho_2}_{L^1}\leq\norm{g_1-g_2}_{L^1},\qquad\text{in}\quad\mathcal{D}'(0,T).
\end{equation*}
\end{thm}
It is important to observe the fact that Theorem \ref{thm:rho m L1-cotraction} holds uniformly to $m\in I$, which further allows us to control the $L^1$ norm for the family of problems $\left\{(P_m')\right\}_{m=2}^{\infty}$ uniformly. This property brings significant convenience in later showing the well-posedness and stability of the posterior distribution of this family of Bayesian inversion.

\subsection{Set up for the prior measure}
\label{sec:Generation and prior measure for the random variables}
In the Bayesian inversion, we shall focus on the models $(P_m')$ and $(P_\infty')$, and treat $u$ as a random variable. In this section we formulate the prior measure of $u$. 

Recall that $u$ contains two different kinds of random quantities, the parametric unknown $z$, and the non-parametric unknown  $h(x)$. For the former, we can assign a prior measure relatively simply. We denote $X_z$ to be the range of $z$, and $\mu_0^z$ be the prior measure of it. 
For a concrete example, considering the case that $z=(z_1, z_2)$ represents the center of the initial data, then we can let the uniform distribution $\mathbb{U}[0,Z_{\max}]^2$ (with some given $Z_{\max}>0$) to be the prior measure $\mu_0^z$, and take $X_z=[0,Z_{\max}]^2$.

However, on the other hand, $h(x)$ is no longer a simple parameter or vector as $z$, but an element in some function space. Therefore, we have to be more careful about selecting the prior measure of it. Fortunately, there is a natural way for setting probability on separable Banach space, in which the elements can be expressed in the form of an infinite series. That is, one can write $h(x)$ into
\begin{equation}
\label{eqn: h(x) representation}
    h(x)=h_0(x)+\sum_{i=1}^{\infty}\gamma_{i}\zeta_i\phi_i,
\end{equation}
where $h_0(x)$ is a deterministic function, $\gamma=\{\gamma_i\}_{i=1}^{\infty}$ is a deterministic sequence of scalars, $\phi=\{\phi_i\}_{i=1}^{\infty}$ is a set of basis functions, and $\zeta=\{\zeta_i\}_{i=1}^{\infty}$ be an i.i.d. random sequence. 
We demonstrate how to select these scalars and functions in the following.

To begin with, we consider the eigen-problem $-\Delta \phi=\lambda \phi$ with Dirichlet boundary condition on $\Omega$. Let $\phi_i$ denote the $i$-th normalized (with respect to $\left\Vert\cdot\right\Vert_{L^{\infty}}$) eigen-function, and $\lambda_i$ be the corresponding eigen-value. Then we make the following assumptions with respect to the expression \eqref{eqn: h(x) representation}.
\begin{assump}
\label{assump:random f and g}
    \begin{enumerate}
        \item $h_0(x)$ is a known positive deterministic function that belongs to the space $L^{\infty}(\Omega)$,
        \item $\gamma=\{\gamma_i\}_{i=1}^{\infty}$ be a deterministic sequence with $\gamma_i=\lambda_i^{-s/2}$ for some $s>1$,
        \item $\zeta=\{\zeta_i\}_{i=1}^{\infty}$ be an i.i.d. random sequence with $\zeta_i\sim N(0,1)$, thus $\zeta$ can be viewed as a random element in the probability space $(\mathbb{R}^{\infty},\mathcal{B}(\mathbb{R}^{\infty}),\mathbb{P})$, where $\mathbb{P}$ denotes the infinity Cartesian product of $N(0,1)$,
        \item $\{\phi_i\}_{i=1}^{\infty}$ denote the normalized eigen-functions of $-\Delta$ as prescribed.
    \end{enumerate}
\end{assump}
Then we let $X_h$ denote the closure of the linear span of the functions $(h_0,\{\phi_i\}_{i=1}^{\infty})$ with respect to the norm $\norm{\cdot}_{L^{\infty}(\Omega)}$. Thus, the Banach space $(X_h,\norm{\cdot}_{L^{\infty}(\Omega)})$ is separable (recall the fact that $L^{\infty}(\Omega)$ is not separable itself). 
Furthermore, with above setting $h(x)$ becomes a sample from the Gaussian measure $\mu_0^h:=N(h_0(x), (-\Delta)^{-s})$. And by a standard argument (see, e.g., Theorem 2.12. in \cite{dashti2017bayesian}), we have $h(x)\in C^{0,t}(\Omega)$ hold $\mu_0^h$-a.s. for any $t<1\wedge(s-1)$. Then, by embedding theory, one can further conclude $h(x)\in L^{\infty}(\Omega)$ hold $\mu_0^h$-a.s..

For convenience, we further define the Banach space for $u$
\begin{equation}
\label{eqn:def space X}
 X:=X_z\times X_h   
\end{equation}
with respect to the norm
\begin{equation}
\label{eqn:upper bound for u}
   \norm{u}_{X}
    \coloneqq\max \left\{\vert z\vert, \norm{h}_{L^{\infty}(\Omega)}\right\}
\end{equation}
where $\left\vert\cdot\right\vert$ denotes the Euclidean distance on $\mathbb{R}^2$. And the prior measure for $u$, $\mu_0$, is given by the product measure
\begin{equation}
\label{eqn:def mu0}
    \mu_0:=\mu_0^z\times\mu_0^h.
\end{equation}
Then one has $\mu_0(X)=1$. 

\subsection{Well-posedness and stability of the inverse problems} 
In this section we establish the well-posedness and stability results for the inverse problems $(P_m')$ and $(P_\infty')$. And we emphasize that these results hold for any $m\in[2,\infty]$, in particular $m=\infty$ corresponds to $(P_\infty')$. 

For the convenience of the reader, we recall the definition of prior measure and the noise vector here:
\begin{itemize}
    \item \textbf{Prior:} $u\sim \mu_0$ measure on $X$, with $X$ and $\mu_0$ defined in \eqref{eqn:def space X} and \eqref{eqn:def mu0} respectively.
    \item \textbf{Noise:} $\eta\sim N(0,\Gamma)$, where $\Gamma$ is a $JK$ by $JK$ diagonal matrix with the diagonal elements given by $\sigma^2_{j,k}>0$. 
    \item \textbf{Noisy observation:} Consider $(P_m')$ with any given $u\in X$, then the noisy observation $y\sim N(\mathcal{G}^m(u),\Gamma):=\mathbb{Q}_0$, where $\mathcal{G}^m$ is defined in \eqref{eqn:noisy observation for m}. Similarly, for $(P_\infty')$ one has $y\sim N(\mathcal{G}^\infty(u),\Gamma)$.
\end{itemize}
For later convenience, we further define the product measure $\nu_0$ to be
\begin{equation}
\label{eqn:nu0}
    \nu_0(du,dy)=\mu_0(du)\mathbb{Q}_0(dy).
\end{equation}
In the following, we mainly focus on the case $(P_m')$ , but one can establish similar results to $(P_\infty')$ without any difficulty. 

Our interest is the posterior distribution of $u$ given $y$, denote as $\mu^{y}_m$. With the prior, noise, and noisy observation above, one can first write out the Radon-Nikodym relation between $\mu_0$ and $\mu^{y}_m$ as follows:
\begin{equation}
\label{eqn:Phi_Z}
    \frac{d\mu_m^{y}}{d\mu_0}(u)=\frac{1}{Z_{m}(y)}\exp{(-\Phi_m(u,y))},\qquad
    Z_{m}(y)\coloneqq\int_{X}\exp{(-\Phi_m(u,y))}d\mu_0(u),
\end{equation}
where the potential function $\Phi_m(u,y)$ is given by:
\begin{equation}
\label{eqn:Phi_m}
    \Phi_m(u,y)=\frac{1}{2}\vert\Gamma^{-1/2}(\mathcal{G}^m(u)-y)\vert^2-\frac{1}{2}\vert\Gamma^{-1/2}y\vert^2.
\end{equation}
And we can define $(\Phi_{\infty},Z_{\infty}, \mu_{\infty}^y)$ analogously for the problem $(P_\infty')$.

Then to justify the well-posedness and stability of the posterior distribution $\mu_m^y$ reduces to the justification of the well-posedness and stability of the Radon-Nikodym relations \eqref{eqn:Phi_Z}. To do this, following the framework in \cite{dashti2017bayesian}, it is sufficient for us to check the following properties for the potential function $\Phi_m$. And parallelly, $\Phi_\infty$ for $\mu_{\infty}^y$.
\begin{prop}
Consider $(P_m')$ with any $m\geq 2$, let $u\sim\mu_0$. Then the potential $\Phi_{m}$ satisfies
\label{prop:props for Phi and Z}
\begin{enumerate}
    \item 
    $\Phi_{m}(u,y)$ is $\nu_0$ measurable (defined in \eqref{eqn:nu0});
    \item 
    there exist function $M_i:\mathbb{R}^{+}\times\mathbb{R}^{+}\mapsto\mathbb{R}^{+}$, $i=1,2$, monotonic non-decreasing, and $M_2$ strictly positive such that for all $u\in X$, $y,y_1,y_2\in \mathbb{B}_{r}(0)\subseteq Y$:
    \begin{align}
        \Phi_m(u,y)&\geq -M_1(r,\norm{u}_{X}),\\
        \vert\Phi_m(u,y_1)-\Phi_m(u,y_2)\vert&\leq M_2(r,\norm{u}_{X})\vert y_1-y_2\vert;
    \label{eqn: difference between potential fcns}
    \end{align}
    \item
    if further 
    \begin{equation}
        \exp{(M_1(r,\norm{u}_X))}\in L_{\mu_0}^1(X;\mathbb{R}),
    \end{equation}
    for any $r>0$. Then the normalization constant $Z_m$ given by \eqref{eqn:Phi_Z} is positive $\mathbb{Q}_0$-a.s..
\end{enumerate}
\end{prop}
\begin{rem}
\label{rmk:prop for Phi infty}
The above proposition hold for $\Phi_{\infty}$ as well. In particular, the second proposition for $\Phi_{\infty}$ hold with the same $M_1$ and $M_2$ as $\Phi_{m}$. This can be see from the proof of Proposition \ref{prop:props for Phi and Z} and Lemma \ref{lem:L1 bound of rho}.
\end{rem}
Before showing the above properties, we establish following auxiliary lemmas first.
\begin{lem}[Lemma 3.3,\cite{dashti2017bayesian}]
\label{lem:measurable lem}
Let $(Z,B)$ be a Borel measurable topology space and assume that $G\in\mathcal{C}(Z;\mathbb{R})$ and that $\pi(Z)=1$ for some probability measure $\pi$ on $(Z,B)$. Then $G$ is a $\pi$-measurable function.
\end{lem}
\begin{lem}
\label{lem:L1 bound of rho}
For $u=(z,h(x))$, with $h(x)$ satisfy Assumption \ref{assump:random f and g}. Let $\rho$ be either $\rho^{(m)}$(any $m\geq 2$) or $\rho^{(\infty)}$ with initial condition $f_0(x+z)$. Then for any $0\leq t\leq T$, we have
\begin{equation}
    \norm{\rho(t)}_{L^1}\leq\pi e^{\norm{u}_{X}T}.
\end{equation}
\end{lem}
\begin{proof}
According to Theorem \ref{thm:rho m L1-cotraction} and Theorem \ref{thm:rho infty L1-cotraction} (set $\rho_1=\rho$, $\rho_2=0$, and $g$ given by \eqref{eqn:nonlocalFKPP g}), in either case we have
\begin{align*}
    \norm{\rho}_{L^1}
    &\leq\norm{f_0(x+z)}_{L^1}+\int_0^T \norm{h(x)\rho}_{L_1} dt\\
    &\leq \pi+\norm{u}_{X}\int_0^T\norm{\rho}_{L^1}dt.
\end{align*}
Finally, we complete the proof by applying Grownwall's inequality.
\end{proof}
With the support of above lemmas, we can easily verify the properties in Proposition \ref{prop:props for Phi and Z}.
\begin{proof}[Proof of Proposition \ref{prop:props for Phi and Z}]
For concision, we omit the superscript $m$ and simply use $\rho$ to denote the density. 

For $(1)$, according to Lemma \ref{lem:measurable lem}, it is sufficient to us to check $\Phi_m(u,y)$ is bounded in each variable. Note that for each component of $\mathcal{G}^m(u)$ we have
\begin{align*}
        l_{j,k}(\rho)
        &=\int_{\Omega}\xi_k(x)\rho(x,t_j)dx\\
        &\leq \norm{\xi_k}_{L^{\infty}(\Omega)}\norm{\rho(t_j)}_{L^1}\\
        &\leq \norm{\xi_k}_{L^{\infty}(\Omega)}\pi e^{\norm{u}_{X}T},
\end{align*}
where we used Lemma \ref{lem:L1 bound of rho}. Thus,
\begin{align}
\label{eqn:upbd for Phi}
    \vert\Phi_m(u,y)\vert
    &=\frac{1}{2}\left\vert\vert\Gamma^{-1/2}(\mathcal{G}^m(u)-y)\vert^2-\vert\Gamma^{-1/2}y\vert^2\right\vert\\\nonumber
    &\leq C\left(\vert\mathcal{G}^m(u)\vert^2+\vert y\vert^2\right)\\\nonumber
    &\leq C\left(e^{2\norm{u}_{X}T}+\vert y\vert^2\right),
\end{align}
Therefore, $\Phi_m(u,y)$ is bounded in each variable and we complete the proof.

For $(2)$, the first inequality hold obviously with
\begin{equation}
\label{eqn: select_M1}
    \Phi_m(u,y)\geq -\frac{1}{2}\left\vert\Gamma^{-1/2}y\right\vert^2\geq -C_{\Gamma}\cdot r^2:=-M_1(r,\norm{u}_{X}),
\end{equation}
where $C_{\Gamma}$ is a constant depend on the covariance matrix $\Gamma$. While, for the second inequality, by using the bounds in part $(1)$, we have
\begin{align*}
    &\vert\Phi_m(u,y_1)-\Phi_m(u,y_2)\vert\\
    &=\frac{1}{2}\left\vert\vert\Gamma^{-1/2}(\mathcal{G}^m(u)-y_1)\vert^2-\vert\Gamma^{-1/2}y_1\vert^2-\vert\Gamma^{-1/2}(\mathcal{G}^m(u)-y_2)\vert^2-\vert\Gamma^{-1/2}y_2\vert^2\right\vert\\
    &\leq C\left(\vert y_1+y_2-2\mathcal{G}^m(u)\vert+\vert y_1+y_2\vert\right)\vert y_1-y_2\vert\\
    &\leq C(r+\pi e^{\norm{u}_X T})\vert y_1-y_2\vert.
\end{align*}
Thus, $M_2(r,\norm{u}_X)$ can be chosen as
\begin{equation*}
    M_2(r,\norm{u}_X)= C(r+\pi e^{\norm{u}_X T})\vert y_1-y_2\vert.
\end{equation*}

For $(3)$, utilizing \eqref{eqn:upbd for Phi}  one can show that for $\mathbb{Q}_0$-a.s.  $\Phi(\cdot,y)$ is bounded on 
\begin{equation}
    X_0=[0,Z_{\max}]^2\times\mathbb{B}_1,
\end{equation}
where $\mathbb{B}_1$ stands for the unit ball in $X_h$. We denote the resulting bound by $M=M(y)$, then
\begin{equation*}
    Z_m\geq \int_{X_0}\exp{(-M)}\mu_0(du)>0,
\end{equation*}
where we used the fact that all balls have positive measure for Gaussian measure on a separable Banach space. 
\end{proof}

Before we establish the formal well-posedness and stability results, we introduce the Hellinger distance.
\begin{defn}
\label{def:H-distance}
Assume $\mu_1$ and $\mu_2$ be two probability measures that both absolutely continuous with respect to $\mu_0$ i.e. $\mu_i\ll\mu_0$ for $i=1,2$, then the Hellinger distance $d_{H}(\mu_1,\mu_2)$ between $\mu_1$ and $\mu_2$ is defined as
\begin{equation*}
    d_{H}(\mu_1,\mu_2)=\left(\frac{1}{2}\int_{X}\left(\sqrt{\frac{d\mu_1}{d\mu_0}}-\sqrt{\frac{d\mu_2}{d\mu_0}}\right)^2 d\mu_0\right)^{1/2}.
\end{equation*}
\end{defn}
Proposition \ref{prop:props for Phi and Z} further yields the following two items.
\begin{thm}[Well-posedness of the posterior distribution]
\label{thm:wellposedness}
Consider the inverse problem of finding $u=(z,h(x))$ from noisy observations of the form \eqref{eqn:noisy observation for m} subject to $\rho^{(m)}$ solving $(P_m')$ ($m\geq 2$), with observational noise $\eta\sim N(0,\Gamma)$. Let $\mu_0$ be the prior measure defined in \eqref{eqn:def mu0} such that $\mu_0(X)=1$, where $X$ is the Banach space defined in \eqref{eqn:def space X}. Then the posterior distribution $\mu_m^{y}$ given by the relation \eqref{eqn:Phi_Z} is a well-defined probability measure.
\end{thm}
\begin{rem}
The well-posedness of the posterior distribution is equivalent to the well-posedness of the Radon-Nikodym relation in \eqref{eqn:Phi_Z}, which has already been checked in Proposition \ref{prop:props for Phi and Z}.
\end{rem}
\begin{thm}[Stability of the posterior distribution]
\label{thm:stability}
With the same set up as in Theorem \ref{thm:wellposedness}, if we additionally assume that, for every fixed $r>0$,
\begin{equation}
\label{eqn: integrable condition}
  \exp{(M_1(r,\norm{u}_X))}(1+M_2(r,\norm{u}_X)^2)\in L_{\mu_0}^1(X;\mathbb{R}).  
\end{equation}
Then  there exists a positive constant $C(r)$ such that for all $y_1,y_2\in\mathbb{B}_r(0)\subseteq Y$
\begin{equation*}
    d_{H}(\mu_m^{y_1},\mu_m^{y_2})\leq C\vert y_1-y_2\vert.
\end{equation*}
\end{thm}
 Regarding the integrability condition \eqref{eqn: integrable condition}, it is worth noting that the function $M_1$ can be chosen independent of $\Vert u\Vert_X$ as specified in \eqref{eqn: select_M1}. Thus, one can apply the Fernique theorem (see Theorem 7.25 in \cite{dashti2017bayesian}) to obtain \eqref{eqn: integrable condition}.
 
 The proof of Theorem \ref{thm:stability} is standard (see Section 4 of \cite{dashti2017bayesian}), so we only describe the main idea but without providing a detailed proof. By a direct calculation, $d_{H}(\mu_m^{y_1},\mu_m^{y_2})$ can be present as an integral in terms of $\vert\Phi_m(u,y_1)-\Phi_m(u,y_2)\vert$ with respect to the prior measure. Then one can complete the proof by applying estimate in \eqref{eqn: difference between potential fcns} and the integrable condition \eqref{eqn: integrable condition}.

 Finally, we remark that according to Remark \ref{rmk:prop for Phi infty}, the above two theorems (well-posedness and stability) hold for problem $(P_\infty')$ similarly.

\subsection{Convergence of the posterior distribution}
In this section, to further exclude the possibility that the posterior distribution for $(P_m')$ diverges as $m$ tends to infinity, we show that $\mu_m^y$ indeed converges to $\mu_{\infty}^y$ in the sense of the Hellinger distance. The formal statement is presented in Theorem \ref{thm:convergence of posterior} below. And we emphasize that the incompressible limit of the forward problems yields pointwise convergence of the potential function $\Phi_m(u,y)$, which plays a crucial role in the proof of Theorem \ref{thm:convergence of posterior}.

\begin{thm}
\label{thm:convergence of posterior}
For any $y\in Y$ and $u\sim\mu_0$, let $\mu_m^{y}$ and $\mu_{\infty}^{y}$ be the posterior distribution respect to $(P_m')$ and $(P_\infty')$, then
\begin{equation}
    d_{H}(\mu_m^{y},\mu_{\infty}^y)\rightarrow 0\qquad\text{as}\quad m\rightarrow\infty.
\end{equation}
And for any $\epsilon>0$, there exists $M>0$ such that 
\begin{equation}
    d_{H}(\mu_{m_1}^{y},\mu_{m_2}^y)<\epsilon,\quad\text{for any}\quad m_1,m_2>M.
\end{equation}
\end{thm}
\begin{cor}
Given $y_1,y_2\in Y$, there exists $M>0$ such that 
\begin{equation*}
    d_{H}(\mu_{m_1}^{y_1},\mu_{m_2}^{y_2})<C\vert y_1-y_2\vert ,\quad\text{for any}\quad m_1,m_2>M.
\end{equation*}
\end{cor}
The above corollary directly follows from Theorem \ref{thm:stability} and Theorem \ref{thm:convergence of posterior} with triangle inequality. 
Now we turn to the proof of Theorem \ref{thm:convergence of posterior}, we first show that the convergence of the forward problem yields pointwise convergence of the potential function $\Phi_m(u,y)$.
\begin{lem}
\label{lem:convergence of potential}
For any $u\in X$ and $y\in Y$, let $\Phi_m$ and $\Phi_{\infty}$ be the potential functions for $(P_m')$ and $(P_\infty')$ defined in \eqref{eqn:Phi_m}, then
\begin{equation*}
     \lim_{m\rightarrow\infty}\vert \Phi_m(u,y)-\Phi_{\infty}(u,y)\vert = 0.
\end{equation*}
\end{lem}
\begin{proof}
Direct compute the difference between $\Phi_m(u,y)$ and $\Phi_{\infty}(u,y)$ to get
\begin{align*}
    \vert\Phi_m(u,y)-\Phi_{\infty}(u,y)\vert
    &=\frac{1}{2}\vert\Gamma^{-1/2}(y-\mathcal{G}^m(u))\vert^2-\frac{1}{2}\vert\Gamma^{-1/2}(y-\mathcal{G}^\infty(u))\vert^2\\
    &\leq C\vert 2y-\mathcal{G}^m(u)-\mathcal{G}^\infty(u)\vert\cdot\vert\mathcal{G}^m(u)-\mathcal{G}^\infty(u)\vert\\
    &\leq C(\vert y\vert+\pi e^{\norm{u}_X T})\vert\mathcal{G}^m(u)-\mathcal{G}^\infty(u)\vert.
\end{align*}
Observe that for each component of $\vert\mathcal{G}^m(u)-\mathcal{G}^\infty(u)\vert$ we have 
\begin{align*}
    \vert l_{j,k}(\rho^{(m)}(u))-l_{j,k}(\rho^{(\infty)}(u))\vert
    &\leq \int_{\Omega}\left\vert\xi_k(x)\left(\rho^{(m)}(x,t_j)-\rho^{(\infty)}(x,t_j)\right)\right\vert dx\\
    &\leq \norm{\xi_k}_{L^{\infty}(\Omega)}\norm{\rho^{(m)}(\cdot,t_j)-\rho^{(\infty)}(\cdot,t_j)}_{L^1(\Omega)}
\end{align*}
Thus by Theorem \ref{thm:L1 convergence thm} part $(1)$, we can conclude 
\begin{equation*}
\label{eqn:limit of Phi_m and P_infty}
    \lim_{m\rightarrow\infty}\vert\Phi_m(u,y)-\Phi_{\infty}(u,y)\vert=0.
\end{equation*}
\end{proof}
We now proceed to the proof of Theorem \ref{thm:convergence of posterior}. We would like to clarify that the proof is similar to that of stability (see Theorem 4.5 in \cite{dashti2017bayesian}). We emphasize the differences here. In the proof of stability, one needs to estimate the difference between $\vert\Phi_m(u,y_1)-\Phi_m(u,y_2)\vert$ and via check the integrability condition \eqref{eqn: integrable condition} to complete the proof. However, in the proof of Theorem \ref{thm:convergence of posterior}, one obtains a sequence of probability integrals involved with $\vert\Phi_m(u,y)-\Phi_{\infty}(u,y)\vert$, which possess a uniform upper bound with respect to $m$. Therefore, one can direct complete the proof by applying Lemma \ref{lem:convergence of potential} and the dominant convergence theorem.
\begin{proof}[Proof of Theorem \ref{thm:convergence of posterior}]
Let $Z_m(y)$ and $Z_{\infty}(y)$ denote the normalization constants for $\mu_m^{y}$ and $\mu_{\infty}^y$ so that 
\begin{align*}
    Z_m &=\int_{X}\exp{\left(-\Phi_m(u,y)\right)}\mu_0(du)>0,\\
    Z_{\infty} &=\int_{X}\exp{\left(-\Phi_{\infty}(u,y)\right)}\mu_0(du)>0.
\end{align*}
we checked $Z_m>0$ in Proposition \ref{prop:props for Phi and Z}, and the strict positivity of $Z_\infty$ can be shown in a similar way. Let $\Tilde{\Phi}_m(u,y)$ denote the positive part of $\Phi_m(u,y)$ in \eqref{eqn:Phi_m}, that is
\begin{equation}
    \Tilde{\Phi}_m(u,y)=\frac{1}{2}\vert\Gamma^{-1/2}(\mathcal{G}^m(u)-y)\vert^2>0,
\end{equation}
and define $\Tilde{\Phi}_{\infty}(u,y)$ similarly. Let $\one_{\text{E}}$ denote the indicator function for the event $E$. Then by a direct calculation we get 
\begin{align*}
    \vert Z_m-Z_{\infty}\vert
    &\leq\exp{(\frac{1}{2}\vert\Gamma^{-1/2}y\vert^2)}\int_{X}\left\vert\exp{(-\Tilde{\Phi}_m)}-\exp{(-\Tilde{\Phi}_{\infty})}\right\vert \mu_0(du)\\
     &\leq C\int_{X}\left(\one_{\vert\Tilde{\Phi}_{m}-\Tilde{\Phi}_{\infty}\vert\leq 1}+\one_{\vert\Tilde{\Phi}_{m}-\Tilde{\Phi}_{\infty}\vert>1}\right)\left\vert\exp{(-\Tilde{\Phi}_m)}-\exp{(-\Tilde{\Phi}_{\infty})}\right\vert \mu_0(du)\\
    &\leq C\int_{X}\one_{\vert\Tilde{\Phi}_{m}-\Tilde{\Phi}_{\infty}\vert\leq 1}\cdot\exp{(-\Tilde{\Phi}_{\infty})}\cdot\Big\vert\exp{(-(\Tilde{\Phi}_{m}-\Tilde{\Phi}_{\infty}))}-1\Big\vert\mu_0(du)\\
    &\quad+C\int_X\one_{\vert\Tilde{\Phi}_{m}-\Tilde{\Phi}_{\infty}\vert>1}\cdot\left\vert\exp{(-\Tilde{\Phi}_m)}-\exp{(-\Tilde{\Phi}_{\infty})}\right\vert\mu_0(du)\\
    &\leq C\int_{X}\one_{\vert\Tilde{\Phi}_{m}-\Tilde{\Phi}_{\infty}\vert\leq 1}\cdot\exp{(-\Tilde{\Phi}_{\infty})}\left(\vert\Tilde{\Phi}_{m}-\Tilde{\Phi}_{\infty}\vert+O(\vert\Tilde{\Phi}_{m}-\Tilde{\Phi}_{\infty}\vert^2)\right)\mu_0(du)\\
    &\quad+C\int_X\one_{\vert\Tilde{\Phi}_{m}-\Tilde{\Phi}_{\infty}\vert>1}\cdot\left(\exp{(-\Tilde{\Phi}_{m})}+\exp{(-\Tilde{\Phi}_{\infty})}\right)\mu_0(du)\\
    &\coloneqq\mathcal{P}_1+\mathcal{P}_2.
\end{align*}
Note that by using the fact that $\Tilde{\Phi}_{\infty}$ and $\Tilde{\Phi}_{m}$ are both positive, one can easily check $\mathcal{P}_1$ and $\mathcal{P}_2$ are both integrable, and possess uniform upper bounds with respect to $m$. Thus by the dominated converge theorem (DCT) and Lemma \ref{lem:convergence of potential}, we get
\begin{equation}
\label{eqn:zm to z_infty}
    \lim_{m\rightarrow\infty}\vert Z_m-Z_{\infty}\vert = 0.
\end{equation}
Since both $\mu_m^{y}$ and $\mu_{\infty}^{y}$ are absolutely continuous with respect to $\mu_0$, by the definition of Hellinger distance we have
\begin{equation*}
    \left( d_{H}(\mu_m^{y},\mu_{\infty}^y)\right)^2\leq I_m^1+I_m^2,
\end{equation*}
where
\begin{align*}
    I_m^1&=\frac{1}{Z_m}\int_{X}\left(\exp{(-\frac{1}{2}\Phi_m(u,y))}-\exp{(-\frac{1}{2}\Phi_{\infty}(u,y))}\right)^2 \mu_0(du),\\
    I_m^2&=\vert Z_m^{-1/2}-Z_{\infty}^{-1/2}\vert^2 \int_{X}\exp{(-\Phi_{\infty}(u,y))} \mu_0(du).
\end{align*}
By using a similar argument used to show \eqref{eqn:zm to z_infty}, one can also split the integral $I_m^1$ into the sets where $\vert \Tilde{\Phi}_{m}-\Tilde{\Phi}_{\infty}\vert\leq 1$ and $\vert \Tilde{\Phi}_{m}-\Tilde{\Phi}_{\infty}\vert>1$. Then by using the fact that $\Tilde{\Phi}_{\infty}$ and $\Tilde{\Phi}_{m}$ are both positive, one can apply DCT to show $\lim_{m\rightarrow\infty}I_m^1=0$ similarly. And for $I_m^2$, we have
\begin{equation*}
    \lim_{m\rightarrow\infty}I_m^2\leq\lim_{m\rightarrow\infty}\left(Z_{m}^{-3}\vee Z_{\infty}^{-3}\right)\vert Z_m-Z_{\infty}\vert^2=0.
\end{equation*}
By now, we have completed the proof of the first part of Theorem \ref{thm:convergence of posterior}, and the second part directly follows from the triangle inequality.
\end{proof}

\section{Numerical Experiments}
\label{sec: Numerical Experiments}

In this section, we aim to carry out systematic numerical experiments to illustrate the properties of the unified numerical method for the Bayesian inversion problems that we have constructed. In particular, we aim to show that the method is able to produce uniformly accurate parameter inferences with respect to the physical index $m$ and the noise level $\sigma$ as well as a quantitative study of the numerical error with various sample sizes.


\subsection{Numerical tests setting}

In our numerical experiments, we consider the tumor growth model in 2D: 
\begin{equation}
\left\{ 
\begin{array}{ll}
\partial_t \rho + \nabla \cdot (\rho  \mathbf{v}) = h( \mathbf{x})\rho, \qquad \mathbf{x} \in \Omega=[a,b]\times [a,b], \\[4pt]
\rho(\mathbf{x},0) = \rho_0(\mathbf{x}), \qquad \mathbf{x} \in \Omega, 
\end{array}
\right.
\end{equation}
with no-flux boundary condition $\rho \mathbf{v} = 0$ for $\mathbf{x}\in \partial\Omega$. 
Here $\mathbf{v}$ is determined by the gradient of the pressure 
$p(\mathbf{x},t)$ which is related to a power of density $\rho(\mathbf{x},t)$, precisely
$$ \mathbf{v} = - \nabla p, \qquad p = \frac{m}{m-1}\rho^{m-1}, \quad, m>1. $$

We first introduce how we measure the accuracy of our numerical algorithm. As an illustrative example, let $u$ be the parameter of interest and the posterior samples generated from the Metropolis-Hastings MCMC method is denoted by $\{u_i\}_{i=1}^N$, with $N$ the sample size after $25 \%$ of burn-in phase (we denote $M$ below as the sample size before the burn-in phase). Since the MCMC approach is a sampling method, we need to repeatedly run the simulation and take the average, in order to improve the accuracy of the algorithm. Set the simulation runs to be $K$ ($K=15$ in our tests), then we estimate the expected value of posterior $h$ by 
\begin{equation} \mathbb{E}(\bar u) \approx \frac{1}{K} \sum_{k=1}^K \bar{u}^{(k)} = \frac{1}{K} \frac{1}{N} \sum_{k=1}^K \sum_{i=1}^N  u_i^{(k)}, 
\end{equation}
where $\{u_i^{(k)}\}_{i=1}^N$ are the posterior samples obtained by $k$-th simulation run for the MCMC algorithm, and 
\begin{equation}\label{G_bar}
\bar u^{(k)} = \frac{1}{N} \sum_{i=1}^N u_i^{(k)}
\end{equation}
is the corresponding estimator for the mean value. 
To compare the distance between $\mathbb{E}(\bar u)$ and the true data $u^{\ast}$ which is assumed known, the mean squared error is evaluated as the following: 
$$ \text{MSE} := \mathbb{E}\left[ (\bar u - u^{\ast})^2 \right] \approx \frac{1}{K} \sum_{k=1}^K \left(\bar u^{(k)} - u^{\ast}\right)^2. $$

\subsection{Numerical experiments}
\subsubsection{Test 1}

In this test, we assume that the growth rate $h$ is spatially homogeneous, and it is only the unknown parameter to be inferred. Let the computational domain be
$\Omega = [-2.2,2.2]\times [-2.2,2.2]$, set the spatial step $\Delta x = \Delta y = 0.1$ and temporal step $\Delta t=0.005$. 
In all of our tests, the Gaussian noise is assumed to follow the distribution $N(0,\sigma^2)$, and we set 
$m=40$ unless otherwise specified. 

\vspace{2mm}

\noindent\textbf{Test 1 (a)}
Consider the initial data 
\begin{equation}
\rho(x,y,0) = \begin{cases}
                  0.9, & \sqrt{x^2 + y^2} - 0.5 - 0.5\sin(4\arctan(\frac{y}{x})) <0,  \\[4pt]
                  0, & \text{otherwise}, 
                  \end{cases} 
\end{equation}
for $(x,y)\in\Omega$. 
Assume the prior distribution for the constant growth rate $h$ is the Gaussian distribution $N(\mu, c_0^2)$ with $\mu=c_0=0.5$. 
Let the true $h^{\ast}=1$, and the observation data be the density at time $T=0.5$ added by the Gaussian noise. 

\begin{table}[htbp]
\begin{tabular}{c c c c c c}
\hline
$\sigma$ & 0.05 & 0.1  &  0.2  &  0.4   \\
$\mathbb{E}(h)$ & 1.0042 &  0.9863 &  0.966 & 0.8253   \\
$\text{MSE}(h)$ & 0.0042  & 0.0174   & 0.034  & 0.1747  \\
\hline
\end{tabular}
\caption{Test 1 (a). Errors for different $\sigma$, by using $M=1000$. }
\label{Test1a_table1}
\end{table}

\begin{table}[htbp]
\begin{tabular}{c c c c c c}
\hline
& $M$ & 100 & 200 & 400 & 800  \\
$\text{MSE}(h)$ & $\sigma=0.1$ &  0.4039 & 0.2799 & 0.0688 & 0.0308   \\
$\text{MSE}(h)$ & $\sigma=1$ &  0.3039 & 0.1085 & 0.0486 & 0.0377 \\
\hline
\end{tabular}
\caption{Test 1 (a).  Error convergence with respect to sample size $M$ for $\sigma=0.1$ and $\sigma=1$, respectively. }
\label{Test1a_table2}
\end{table}

\begin{figure}[H]
	\centering
	\includegraphics[width=0.48\textwidth]{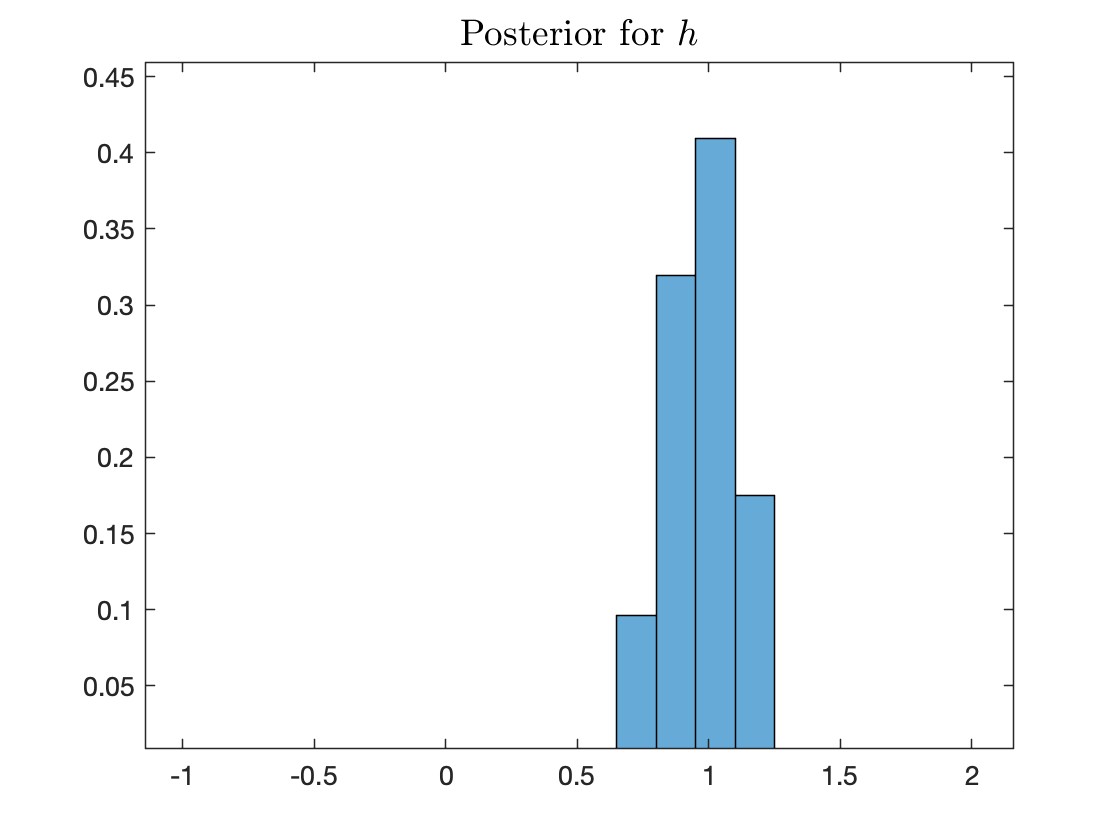}
	\includegraphics[width=0.48\textwidth]{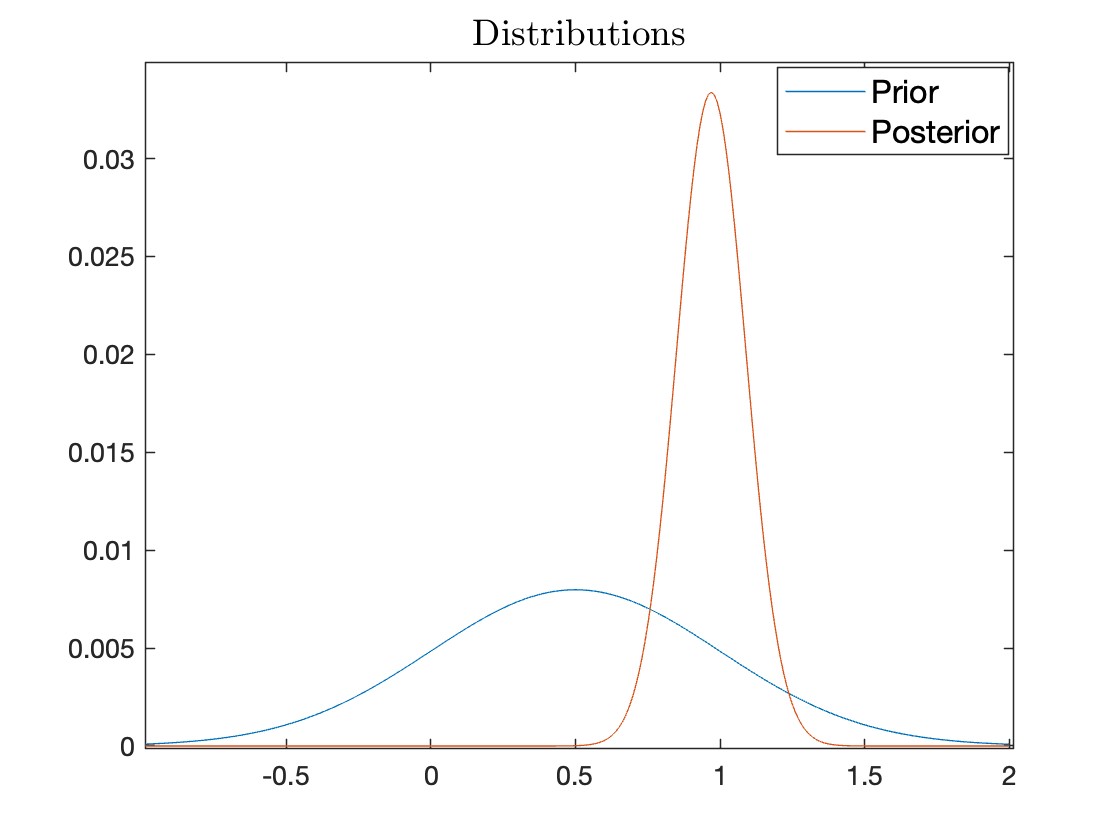}
	\caption{Test 1 (a) with $\sigma=0.1$, by using $M=800$. Left: histogram for the posterior samples. Right: prior and posterior distributions for $h$. }
 \label{Test1a_fig}
\end{figure}
In Table \ref{Test1a_table1}, we fix the physical index $m=40$ and the number of iterations $M=1000$, while letting the noise level $\sigma$ vary. One can observe the accuracy is improved as $\sigma$ decreases, with the level of mean square error of $O(10^{-1})$ to $O(10^{-3})$. This also implies as the noise level is relatively small, the numerical method correctly captures the quantity of interest with satisfactory accuracy. 

In Table \ref{Test1a_table2}, for different $\sigma$ we test by adopting different numbers of sampling iterations $M$. As $M$ increases from $100$ to $800$, the level of mean square errors decreases from $O(10^{-1})$ to $O(10^{-2})$, which is expected due to the decrease of the sampling error.

In Fig. \ref{Test1a_fig}, we plot the histogram for the posterior samples for the parameter $h$ and see how the data is accumulated around the true value $h^{\ast}=1$. A comparison between the prior and posterior distributions for $h$ is shown on the right, with the prior as the Gaussian distribution. 

\vspace{2mm}

\noindent\textbf{Test 1 (b)}  
In this test, we consider the observation data as the density convoluted with Gaussian functions plus noise, which are to model the blurry and noisy observations. 
The centers of the Gaussian functions are chosen to be the grid points $(x_i,y_j)$, where 
$$
\quad i\in \{16, 20, 22, 24, 24, 26, 27, 28, 32\}, \quad j \in \{20, 24, 30, 26, 30, 15, 20, 30, 25\},
$$
and the standard deviation of the noise is $0.1$. The above set of Gaussian functions mimics the local observations of the tumors, i.e., it corresponds to the deterministic test functions $\xi_k$ in equation \eqref{eqn:linear-functional}. The prior distribution for $h$ is assumed as the Gaussian distribution 
$N(\mu, c_0^2)$ with $\mu=c_0=0.5$. Other settings are the same as in Test 1 (a), and we fix the sample size $M=800$ in all tests of Test 1 (b). 

In the following, we further investigate the numerical performance of the proposed method for different physical indexes $m$ and noise levels $\sigma$. In the upper panel of Table \ref{Test1b_table}, we let $m=40$ and test on different  $\sigma$; in the lower panel, we fix $\sigma=0.25$ and make $m$ vary. To help interpret the numerical results,  we plot in Fig. \ref{Test1b_fig} the posterior distributions for different $\sigma$ while fixing $m=40$, and for different $m$  while fixing $\sigma=0.25$.

We observe that from the left panel of Fig. \ref{Test1b_fig} that as $\sigma$ decrease, the posterior distribution contacts to be more peaked while its center is moving towards the true value. And our numerical results give a faithful representation of such a contracting behavior of the posterior distribution: as the variance and the bias of the posterior decreases, the mean squared error of the estimator decreases accordingly.   

When the physical index $m$ changes, we observe from the right panel of Fig. \ref{Test1b_fig} that the posterior does not exhibit a clear trend, however, their profiles do not differ much either. Such an observation confirms our analysis of the convergence behavior of the posterior distributions, and our numerical results also show comparable accuracy although the observation data are actually different for those models. Recall that, given the unknown the forward models generate different results even in the absence of noise. In addition, we have only assumed that the noises added to these models share the same statistical properties.

\begin{table}[htbp]
\begin{subtable}
\centering
\begin{tabular}{c c c c c}
\hline
$\sigma$ & 0.125 & 0.25  &  0.5  &  1   \\
$\text{MSE}(h)$  & 0.00033  & 0.00129  & 0.00591  & 0.02872 \\
\hline
\end{tabular}
\end{subtable}
\hfill
\begin{subtable}
\centering
\begin{tabular}{c c c c c}
\hline
$\text{m}$  &  8  & 16  & 32  & 64   \\
$\text{MSE}(h)$  & 0.0024 & 0.0016 & 0.0010 & 0.0013 \\
\hline
\end{tabular}
\end{subtable}
\caption{Test 1 (b). Errors for different $\sigma$ and $m$. }
\label{Test1b_table}
\end{table}

\begin{figure}[H]
	\centering
	\includegraphics[width=0.5\textwidth]{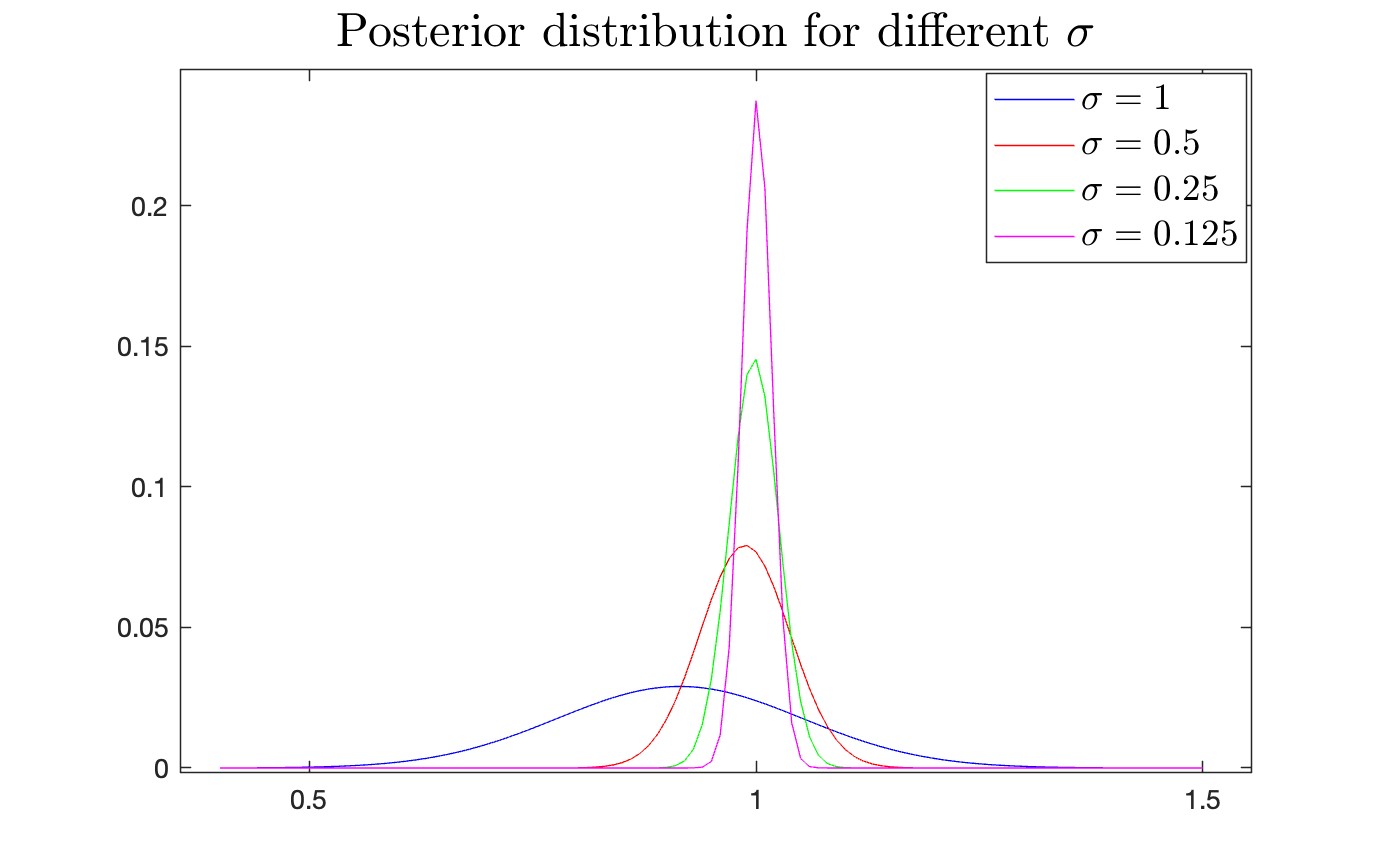}
    \includegraphics[width=0.42\textwidth]{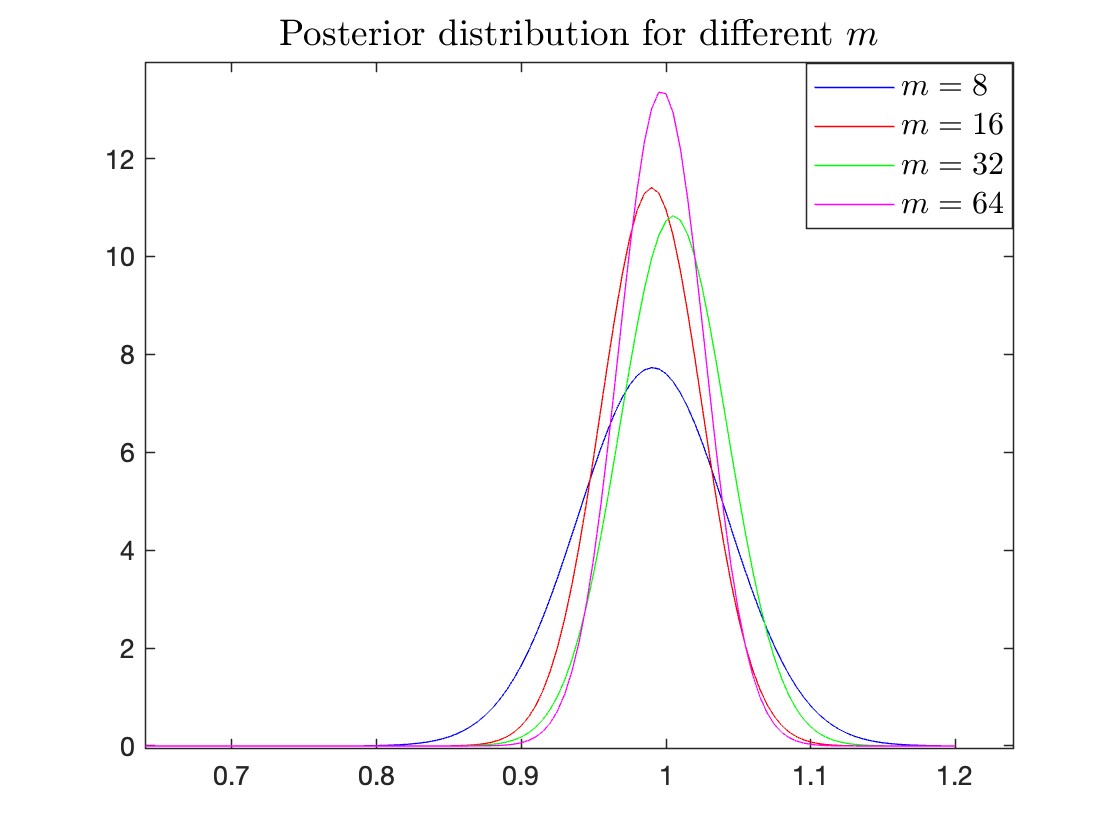}
\caption{Test 1 (b). Posterior distributions of $h$ for different $\sigma$ (fix $m=40$) and different $m$ (fix $\sigma=0.25$). }
\label{Test1b_fig}
\end{figure}

\subsubsection{Test 2}

In Test 2, we consider multi-dimensional random parameters that contain the constant growth rate $h$ and spatial centers of the initial density $c_1$, $c_2$. Let the initial data given by 
\begin{equation}
\rho(x,y,0) = \begin{cases}
                  0.9, & \sqrt{(x-c_1)^2 + (y-c_2)^2} - 0.5 - 0.5 \sin(4\arctan(\dfrac{y-c_2}{x-c_1})) <0,   \\[4pt]
                  0, & \text{otherwise}, 
                  \end{cases} 
\end{equation}
for $(x,y)\in\Omega$. For the prior distributions, we assume the constant growth rate $h$ follow the uniform distribution on $[0.5,0.8]$, while $c_1$ and $c_2$ follow the uniform distribution on $[-0.5, 0.5]$. Let the underlying true data $h^{\ast}=0.6$, $c_1^{\ast}=0.2$, $c_2^{\ast}=-0.3$, and the observation data be the density obtained at final time $T=0.5$, added by the Gaussian noise. In Test 2, we let $M=600$ and $m=40$ unless otherwise specified. 

Note that in this case, the sampling space is three-dimensional, and we can no longer expect the posterior distributions to have simple asymptotic behavior as $\sigma$ or $m$ varies. But still, our results below show that we are able to obtain accurate results for a large range of parameter combinations.  

In the upper panel of Table \ref{Test2_table}, we fix $m=40$ and vary $\sigma$; in the lower panel, we set $\sigma=0.1$ and let $m$ change. In both cases, the mean square errors for $h$ and $c_1$, $c_2$ all remain at the level of $O(10^{-3})$ to $O(10^{-2})$. A similar conclusion can be drawn as before: our algorithm is uniformly accurate with respect to both $\sigma$ and $m$.

In Fig. \ref{Test2_fig}, we plot the histogram of posterior samples for parameters $h$ and $c_1$. One can notice that with a finite noise level $\sigma$, the ``center'' of the distribution for the posterior samples may {\it not} be close to the underlying true data which is given by $h^{\ast}=0.6$ and $c_1^{\ast}=0.2$. Comparing the two examples with $\sigma=0.5$ and $\sigma=0.02$, one can observe that the smaller the $\sigma$ is, the closer and more concentrated the samples are towards the true data for $h$ and $c_1$. 

\begin{table}[H]
\begin{subtable}
\centering
\begin{tabular}{c c c c c c}
\hline
$\sigma$ &  0.0625  & 0.125 & 0.25 & 0.5  & 1 \\
$\text{MSE}(h)$  &  0.0088 &  0.0116  &  0.0236  &  0.0131 & 0.0138 \\ 
$\text{MSE}(c_1)$  &  0.0128 &  0.0157  &  0.0295  &  0.0540 & 0.0432 \\ 
$\text{MSE}(c_2)$  & 0.0039 &  0.0026  &  0.0106  &  0.0476 & 0.0176 \\  
\hline
\end{tabular}
\end{subtable}
\hfill
\begin{subtable}
\centering
\begin{tabular}{c c c c c}
\hline
$\text{m}$  &  8  & 16  & 32  & 64   \\
$\text{MSE}(h)$  &  0.0028 & 0.0039 & 0.0108 & 0.0084 \\
$\text{MSE}(c_1)$  & 0.0262 & 0.0428 & 0.0069 & 0.0460  \\
$\text{MSE}(c_2)$ & 0.0153 & 0.0117 & 0.0523 & 0.0058  \\
\hline
\end{tabular}
\end{subtable}
\caption{Test 2. Errors for different $\sigma$ and $m$. 
}
\label{Test2_table}
\end{table}

\begin{figure}[htbp]
	\centering
	\includegraphics[width=0.48\textwidth]{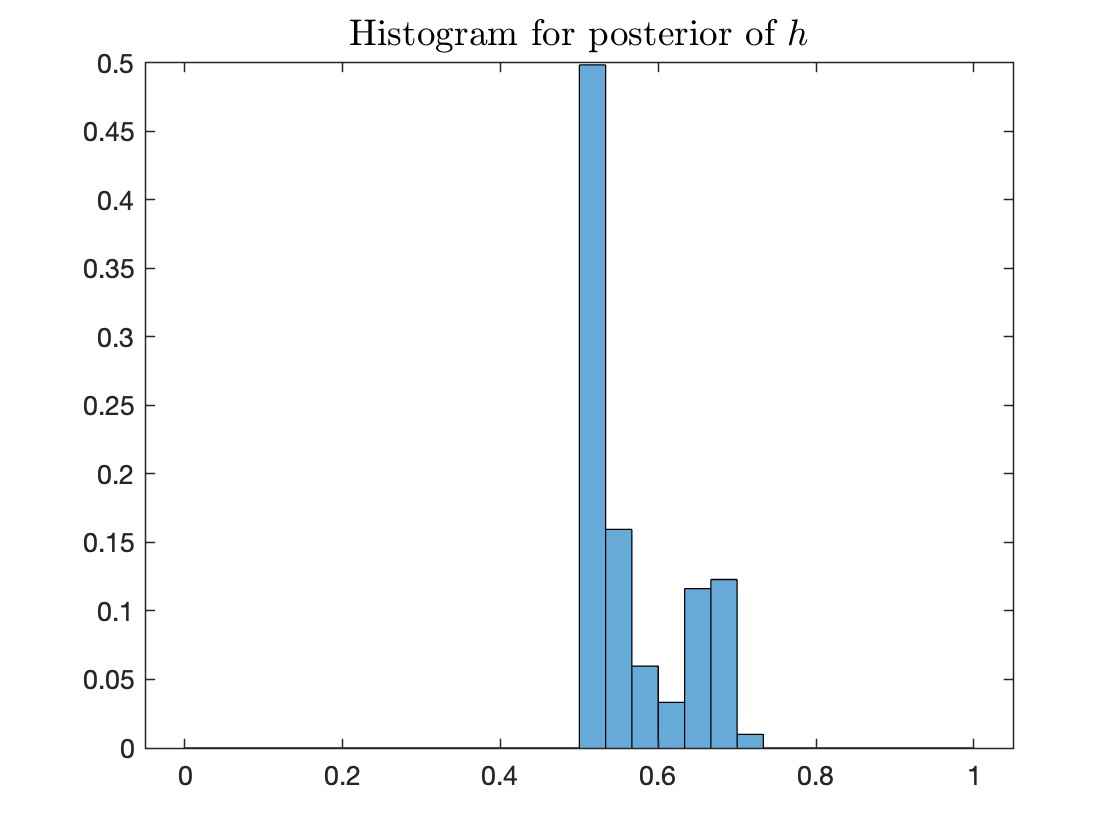}
	\includegraphics[width=0.48\textwidth]{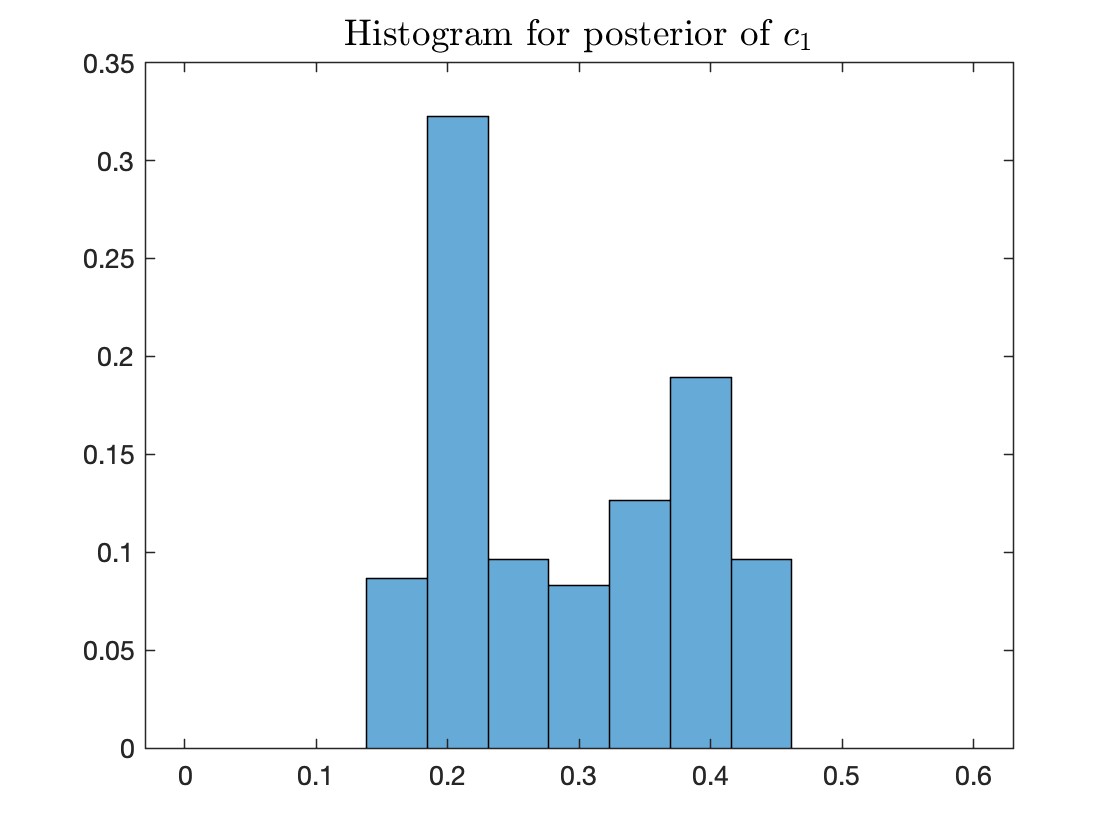}
 \includegraphics[width=0.48\textwidth]{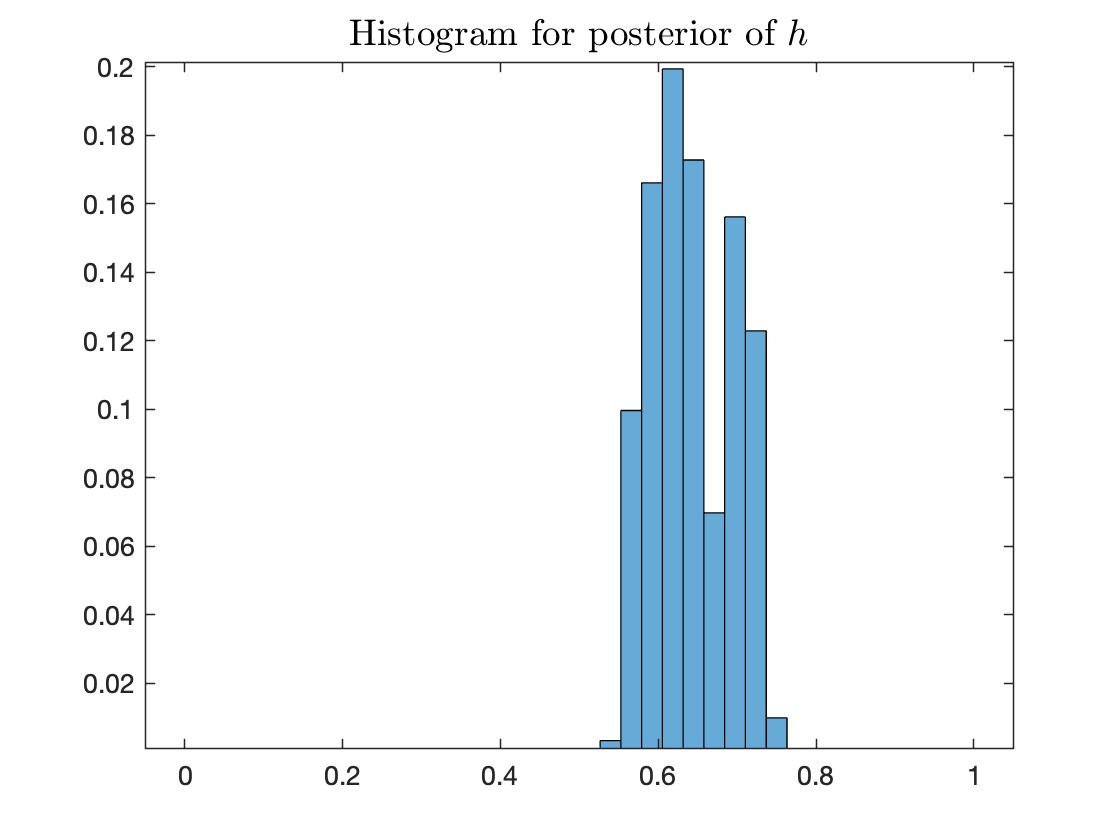}
	\includegraphics[width=0.48\textwidth]{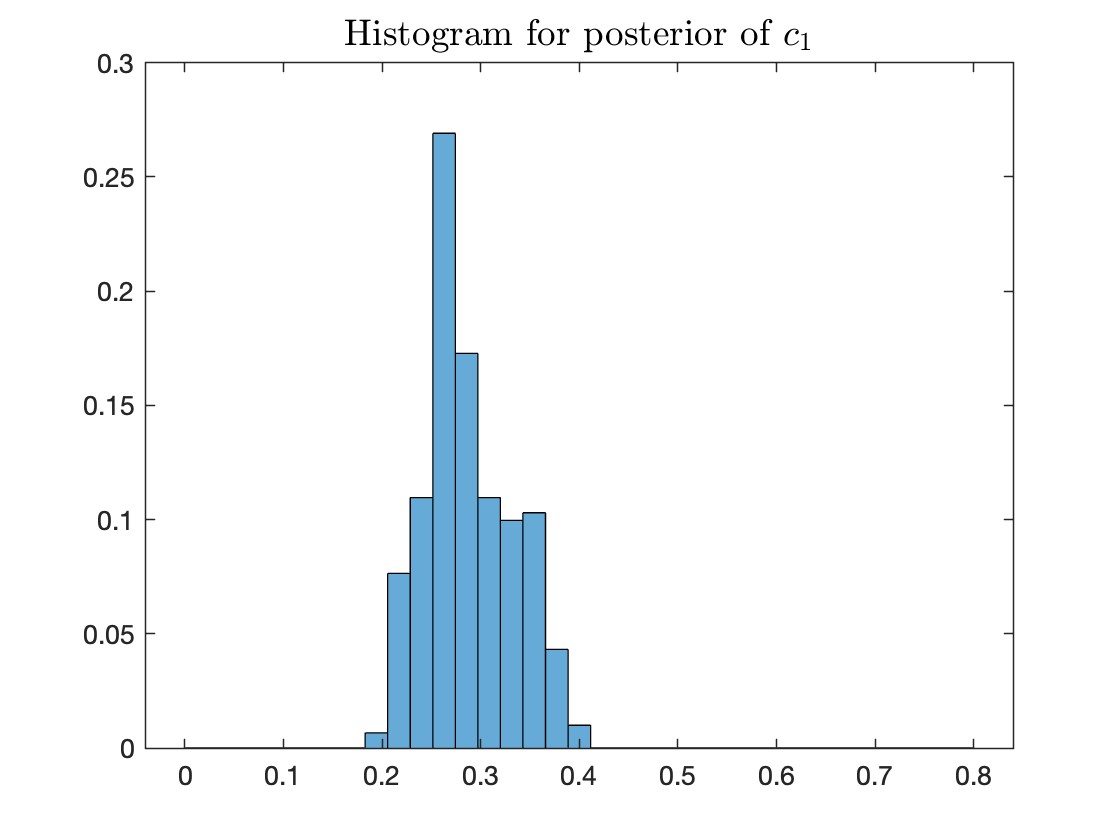}
	\caption{Test 2 with $\sigma = 0.5$ (top) and $\sigma=0.02$ (bottom). Histogram for posterior samples of $h$ and $c_1$.}
	\label{Test2_fig}
\end{figure}
	
\subsubsection{Test 3}

In Test 3, we consider the case when the growth function $h$ is spatially dependent and owns the truncated form of \eqref{eqn: h(x) representation} given by 
$$ h(\mathbf{x}) = h_0(\mathbf{x}) + \sum_{i=1}^3 \gamma_i \zeta_i \phi_i(\mathbf{x}). $$
Let the observation data be the density at the final time $T=0.5$, added by the Gaussian noise. Let the computational domain be $\tilde\Omega = [-1,1]\times[-1,1]$, and $\mathbf{x} = (x, y)\in \tilde\Omega$. We set $h_0=2$ and the initial data given by 
\begin{equation}
\rho(x,y,0) = \begin{cases}
                  0.9, & x^2 + y^2 < 0.2,    \\[4pt]
                  0, & \text{otherwise}. 
                  \end{cases} 
\end{equation} 

In this set up, the first three eigenfunctions $\phi_i$ ($i=1,2,3$) are given by:
\[
\phi_1 (\mathbf{x})= \sin(\pi x_1), \quad \phi_2 (\mathbf{x})= \sin(\pi x_2),\quad \phi_3 (\mathbf{x})= \cos(\pi x_1)\cos(\pi x_1),
\]
and the corresponding eigenvalues are $\gamma_1 = \gamma_2= \frac{1}{\pi^2}$, $\gamma_3 = \frac{1}{2 \pi^2}$. Then we define $g_i := \gamma_i \zeta_i$ and consider $g_i$ as the random variables. More specific, we let the true data for $\zeta$ be $\zeta = (0.8, 0.5, 0.3)$, thus the true data for $g=(0.0811, 0.0507, 0.0152)$, and we assume the prior distribution for $g_i$ follow the Gaussian $N(0, c_i^2)$ with $c_1=0.4$, $c_2=0.3$ and $c_3=0.2$.

In this test, since $h(\mathbf{x})$ is spatially dependent, we approximate the expected value and relative mean squared error by using the following formulas: 
$$ \mathbb{E}(\bar h(\mathbf{x})) \approx \frac{1}{K} \sum_{k=1}^K \bar{h}^{(k)}(\mathbf{x}) = \frac{1}{K}\frac{1}{N} \sum_{k=1}^K  \sum_{i=1}^N  h_i^{(k)}(\mathbf{x}), $$
$$ \text{MSE}: = \mathbb{E} \left[ \|\bar h(\mathbf{x}) - h^{\ast}(\mathbf{x})\|_{L^2}^2\right] \approx \frac{1}{K}\sum_{k=1}^K \|\bar h^{(k)}(\mathbf{x}) - h^{\ast}(\mathbf{x})\|_{L^2}^2, $$
where $h^{\ast}(\mathbf{x})$ is the true data for $h(\mathbf{x})$, shown on the left-hand-side of Fig.\ref{Test3_fig1}, and $\bar h^{(k)}$ is defined in \eqref{G_bar}. In all tests of Test 3, we let the sample size $M=500$.

\begin{table}[htbp]
\begin{subtable}
\centering
\begin{tabular}{c c c c c}
\hline
$\sigma$ & 0.125 & 0.25 & 0.5 & 1 \\
$\text{MSE}(h)$  & 0.0064 & 0.0081 & 0.0076 & 0.0073 \\
\hline
\end{tabular}
\end{subtable}
\hfill
\begin{subtable}
\centering
\begin{tabular}{c c c c c}
\hline
$\text{m}$  &  8  & 16  & 32  & 64   \\
$\text{MSE}(h)$  & 0.0056 & 0.0045  & 0.0034  & 0.0062 \\
\hline
\end{tabular} 
 \end{subtable}
\caption{Test 3. Mean square errors of posterior $h$ for different $\sigma$ and $m$.
 }
\label{Test3_table}
\end{table}

\begin{figure}[h]
	\centering
	\includegraphics[width=0.49\textwidth]{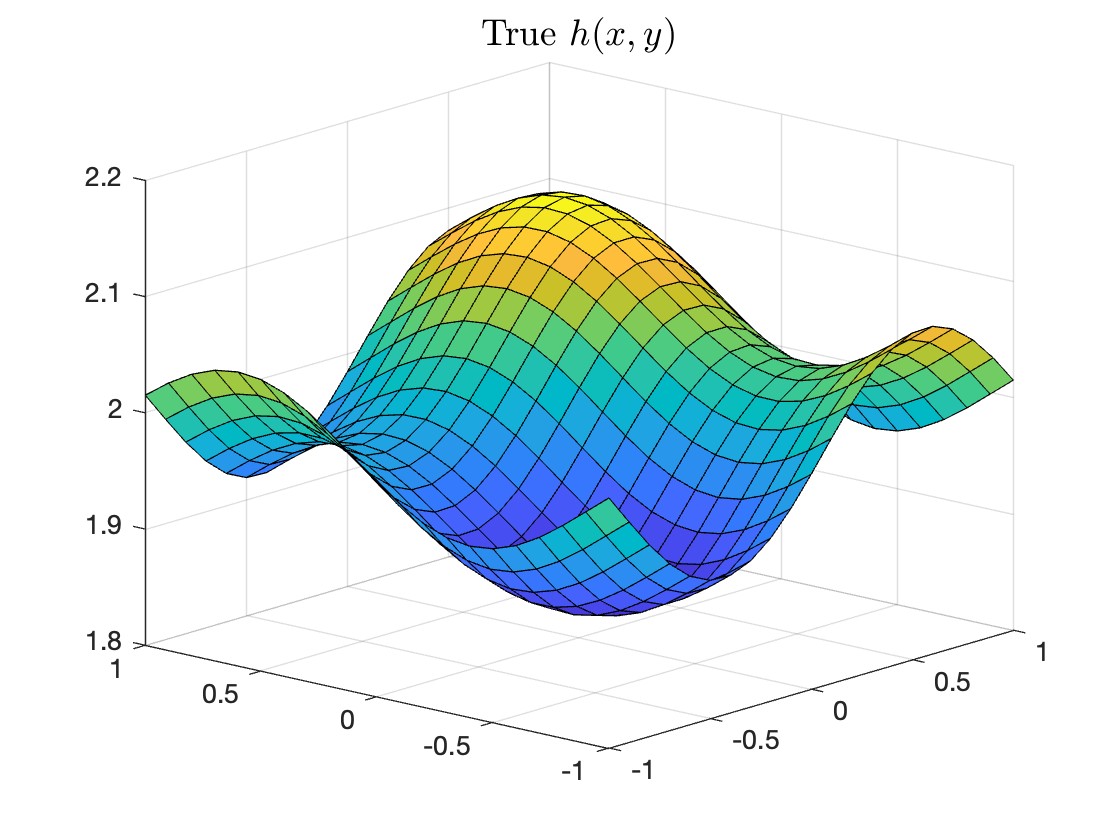}
	\includegraphics[width=0.49\textwidth]{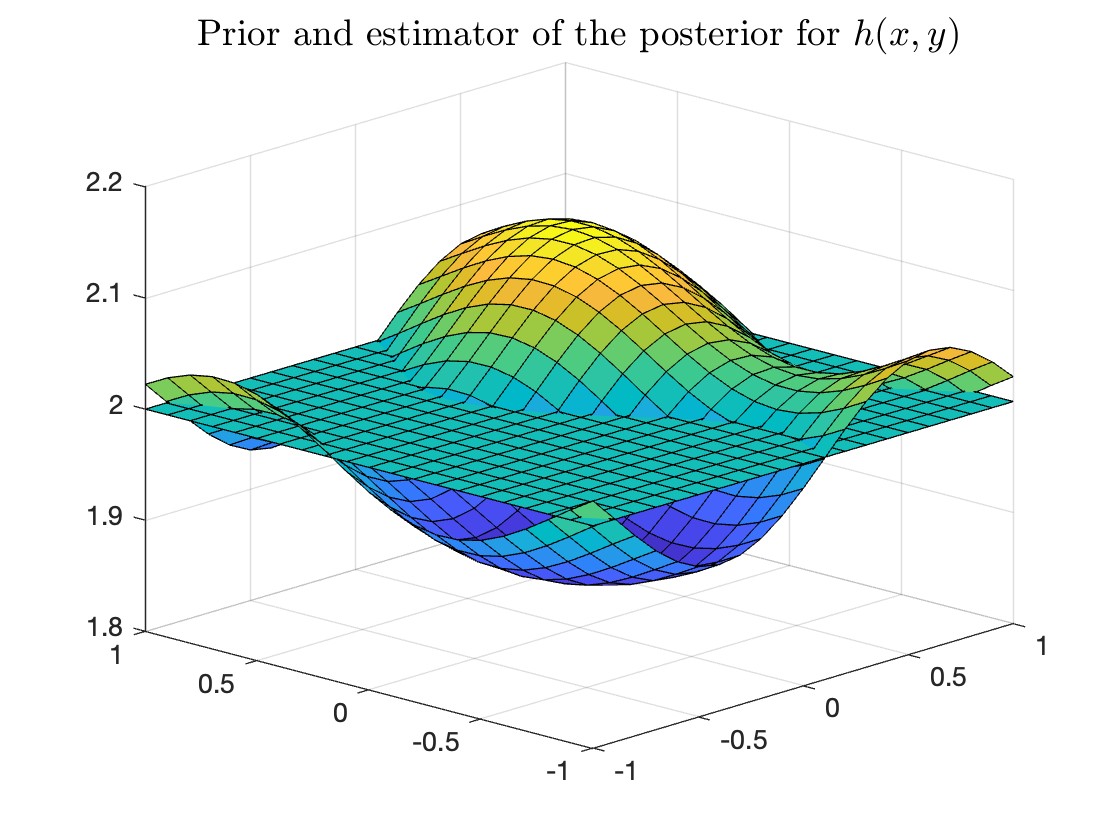}
\caption{ Test 3 with $\sigma=0.25$ and $m=40$. True $h(x,y)$ function, the prior mean (which is $h_0=2$ at all points), and the estimator of the posterior mean of $h(x,y)$ computed pointwisely at each $(x,y)$ point. } 
\label{Test3_fig1}
\end{figure}

\begin{figure}[!h]
\centering
 \includegraphics[width=0.48\textwidth]{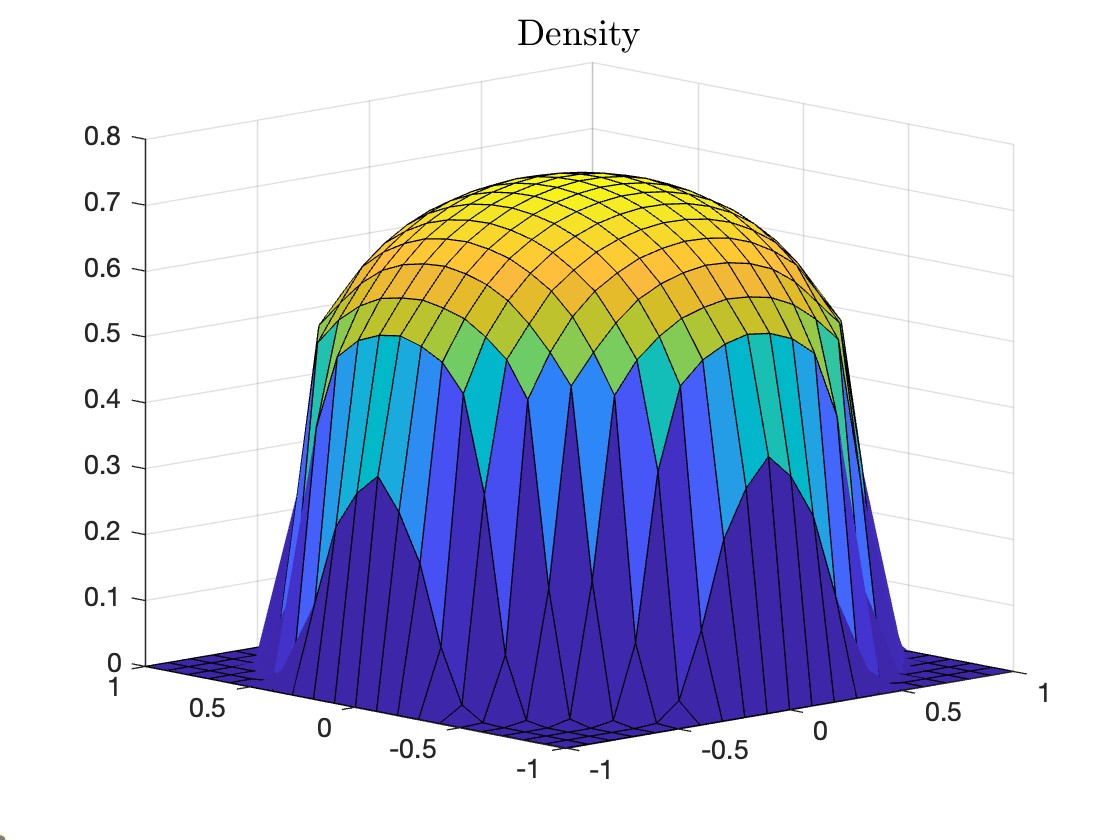}
 \includegraphics[width=0.48\textwidth]{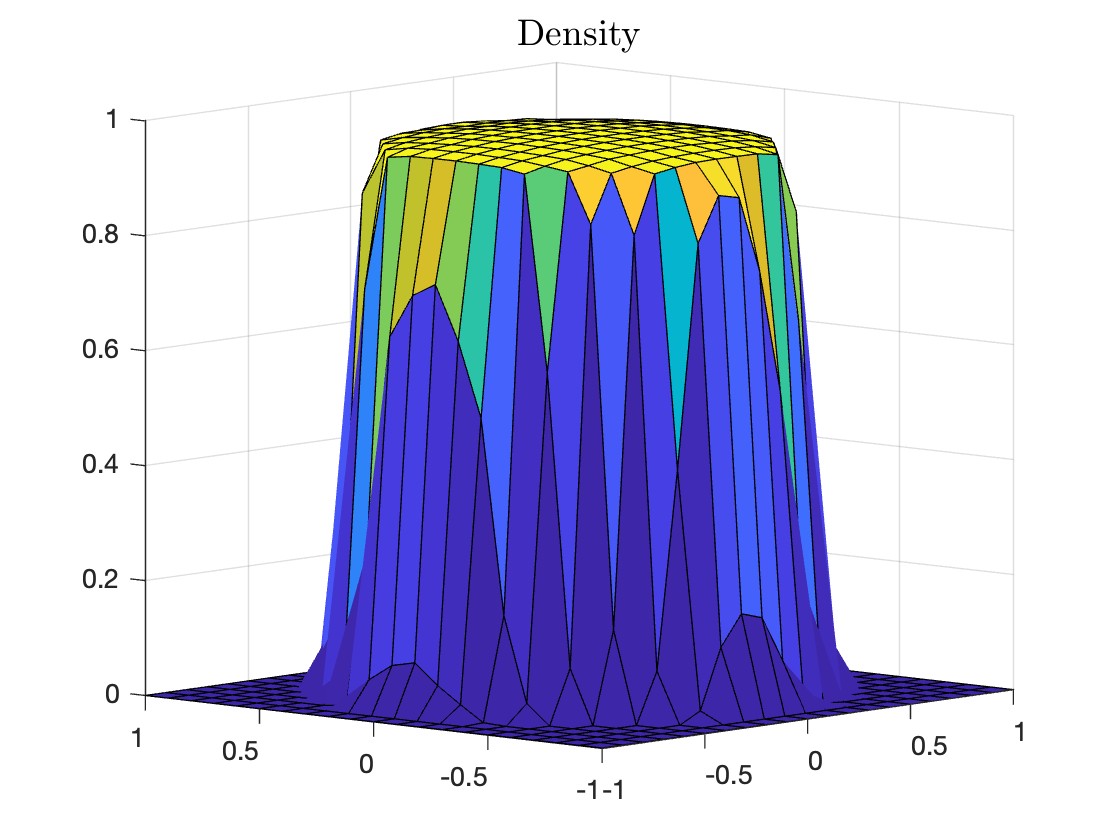}
  \includegraphics[width=0.48\textwidth]{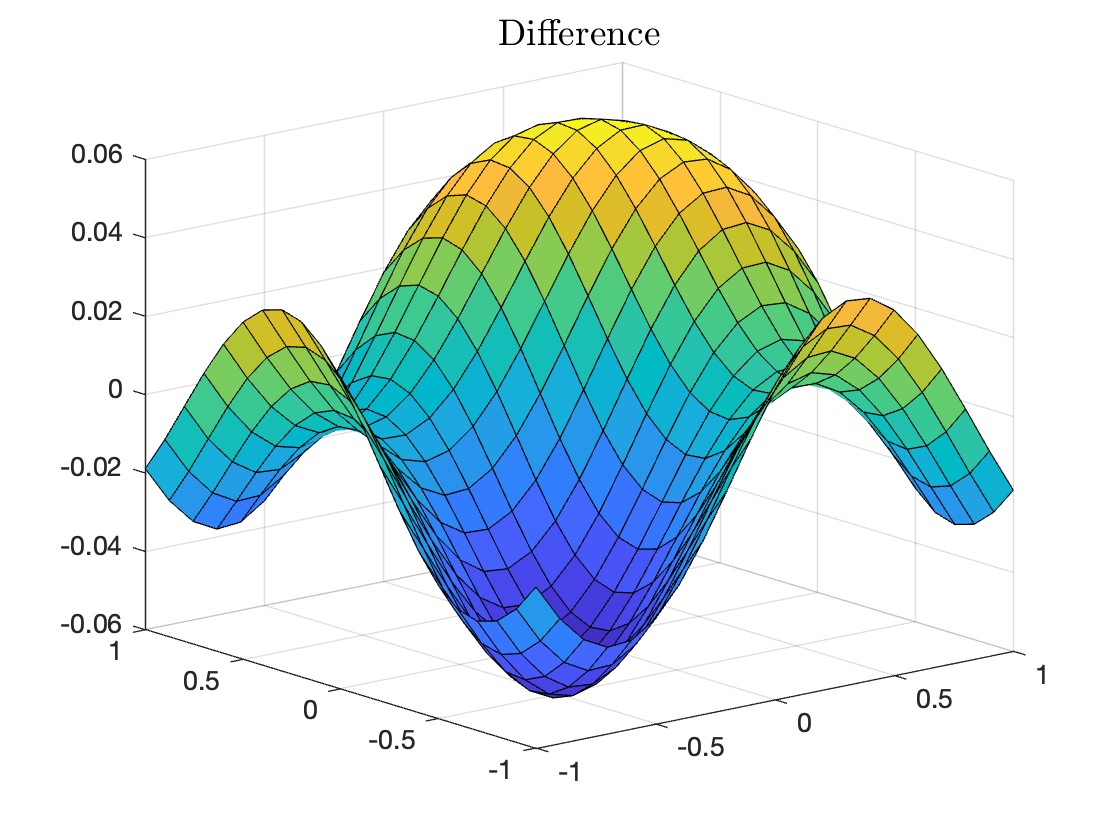}
 \includegraphics[width=0.48\textwidth]{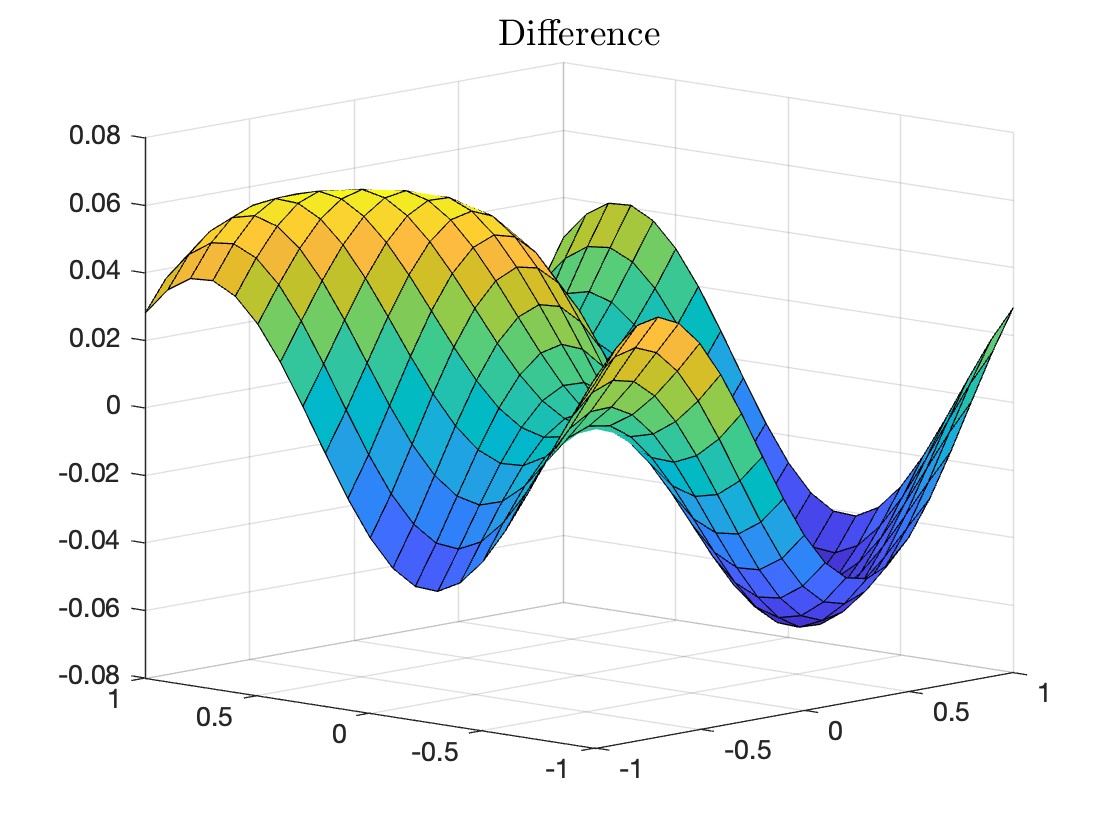}
 \caption{Test 3 with $\sigma=0.5$. Density is computed by using the posterior mean of $h(x,y)$, where $m=5$ (left) and $m=50$ (right). The second row shows the corresponding difference between the posterior mean of $h(x,y)$ and the true function. }
 \label{Test3_fig2}
\end{figure}

\begin{figure}[!h]
\centering
 \includegraphics[width=0.32\textwidth]{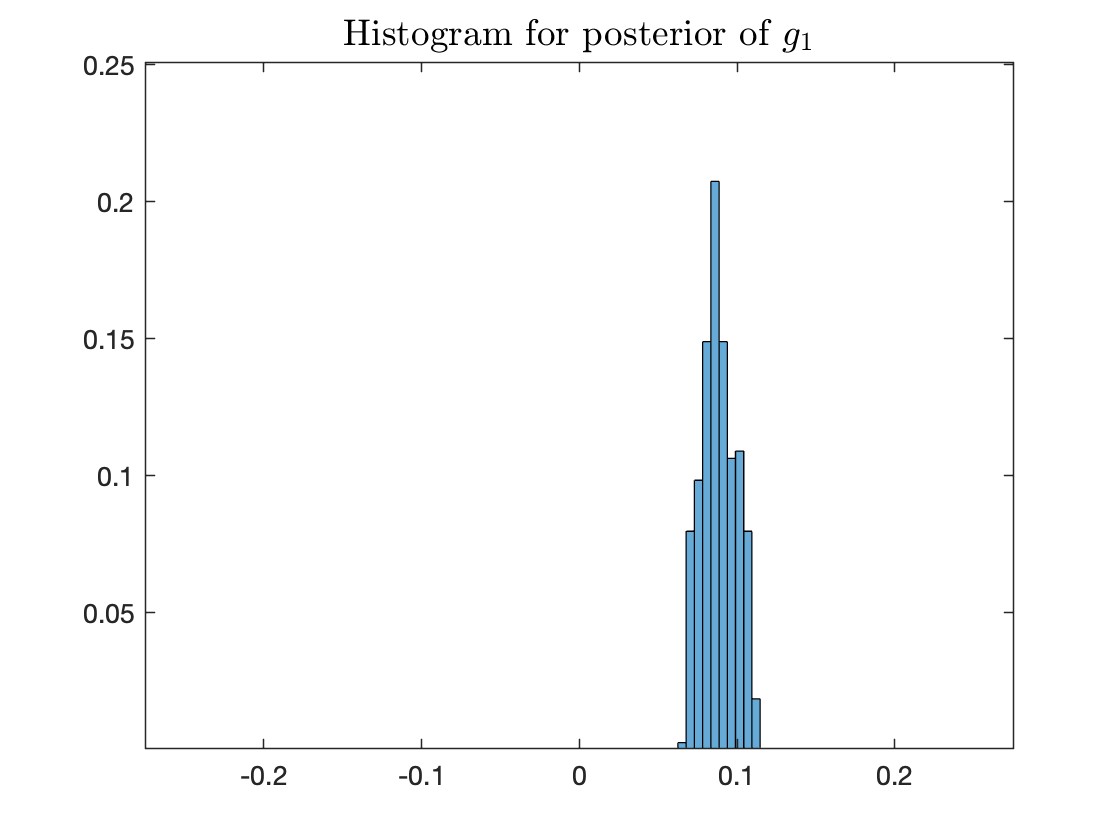}
\includegraphics[width=0.32\textwidth]{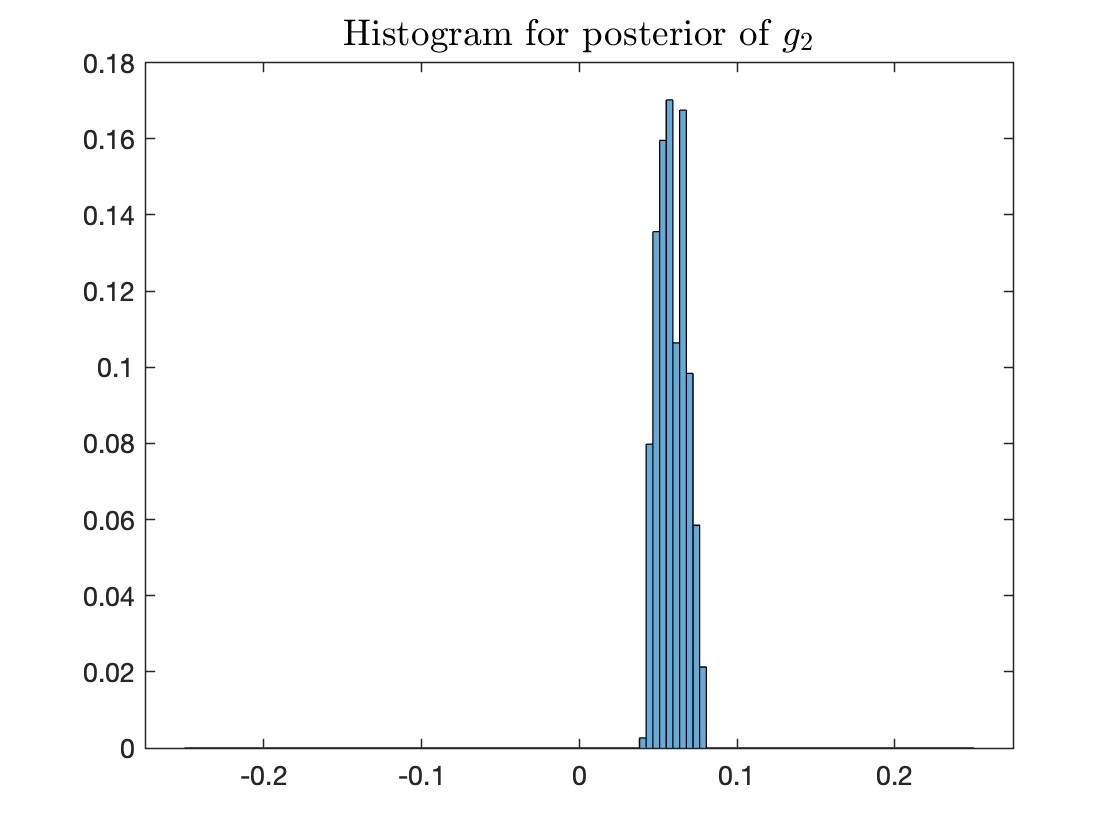}
\includegraphics[width=0.32\textwidth]{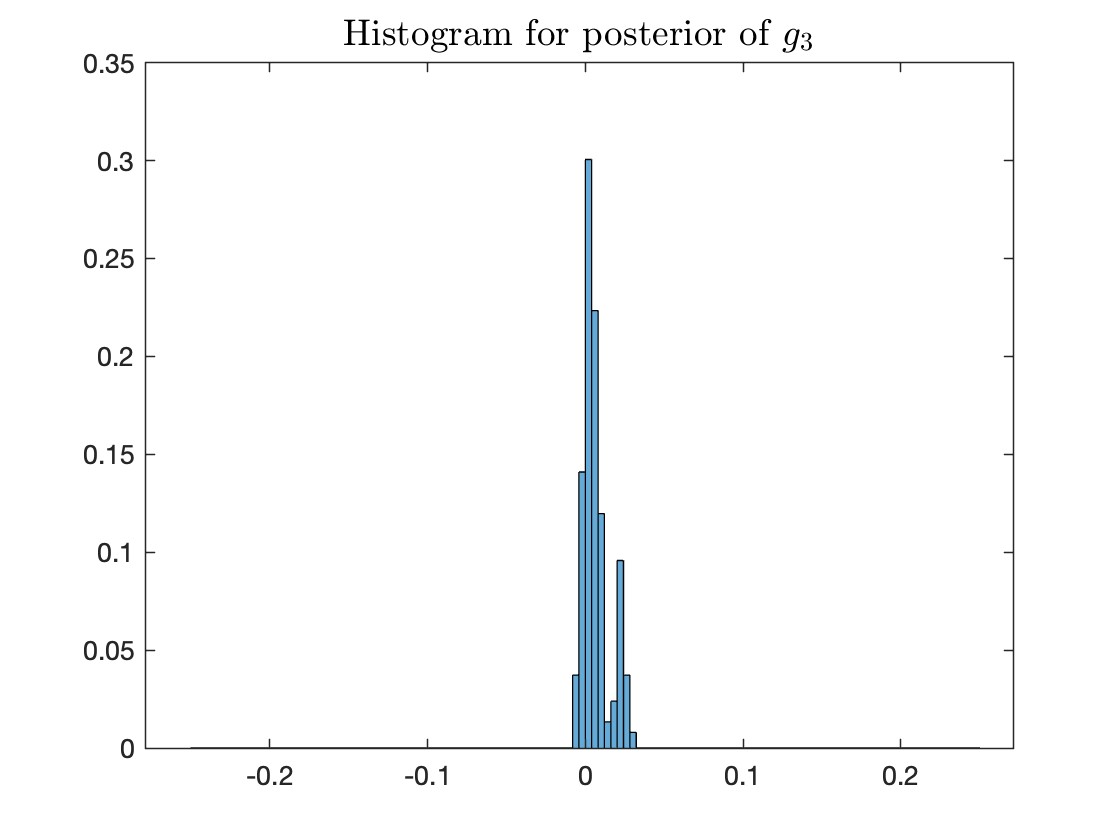}
\caption{ Test 3 with $\sigma=0.125$ and $m=16$. Posterior distributions for $g=(g_1,g_2,g_3)$, compared with the true data $g=(0.0811, 0.0507, 0.0152)$. 
}
 \label{Test3_fig3}
\end{figure}
  
In the upper panel of Table \ref{Test3_table}, we fix $m=40$ and change $\sigma$; in the lower panel, we set $\sigma=0.125$ and let $m$ change. One can observe a uniform accuracy in both cases of varying $m$ and $\sigma$, since the mean square errors remain at the level of as small as $O(10^{-3})$. 

In Fig. \ref{Test3_fig1}, for Test 3 with $\sigma=0.25$ and $m=40$, on the left we plot the true $h(x,y)$ function; on the right we compare the prior and posterior means of $h(x,y)$ which are computed pointwisely at each mesh point $(x,y)$ in the domain. In Fig. \ref{Test3_fig2}, for different choices of $m$ ($m=5$ or 50), we plot the density solution at time $T=0.5$, by using the posterior mean of $h(x,y)$ at each position $(x,y)\in\Omega$. We observe that, with different pressure laws indexed by $m$, the density profiles, as well as their free boundaries, show noticeable discrepancies. However, our numerical method generates accurate inferences of the growth rate functions in both cases as they also deviate from the true data by as small as $O(10^{-2})$ shown in the lower panels of Fig. \ref{Test3_fig2}. 

In Fig. \ref{Test3_fig3}, for Test 3 with $\sigma=0.125$ and $m=16$, we plot the histogram of posterior samples for the parameter $g=(g_1,g_2,g_3)$, comparing with its true data of $g=(0.0811, 0.0507, 0.0152)$. We notice that with a finite noise level $\sigma$ and $m$, the ``centers" of the distribution for the posterior samples of $g$ may not be exactly close to the growth true data, but are acceptably concentrated in a small neighborhood near the true data. 

\section{Conclusion and future work}
In this paper, we investigate the data assimilation problem for a family of tumor growth models that are represented by porous-medium type equations, which is indexed by a physical parameter $m\in[2,\infty]$ characterizing the constitutive relation between the pressure and density. We employ the Bayesian framework to infer parametric and non-

parametric unknowns that affect tumor growth from noisy observations of tumor cell density. We establish the well-posedness and stability theories for the whole family of Bayesian inversion problems. Additionally, to guarantee the posterior has unified behavior concerning the constitutive relations, we further prove the convergence of the posterior distribution in the limit referred to as the incompressible limit, $m \rightarrow \infty$. These theoretical findings guide us in the development of the numerical inference method for the unknowns. We propose a general computational framework for such inverse problems, which encompasses a typical sampling algorithm and an asymptotic preserving solver for the forward problem. We verify through extensive numerical experiments that our proposed framework provides satisfactory and unified accuracy in the Bayesian inference of the family of tumor growth models. 

Finally, we conclude our paper by outlining potential directions for future research. We propose that at least three worthwhile directions merit further exploration. Firstly, we will further employ the real experimental data like that in \cite{falco2023quantifying} for the data assimilation problems of such tumor growth models. Secondly, in this paper, $m$ is assumed to be a known parameter, but it remains interesting to explore the possibility of inferring the index $m$ as well as other unknowns in the model. Thirdly, we may study the Bayesian inversion for other problems that possess nontrivial asymptotic limits. We save these topics for future studies.

\section*{Acknowledgments}
The work of Y.F. is supported by the National Key R\&D Program of China, Project Number 2021YFA1001200. The work of L.L. is supported by the National Key R\&D Program
of China, Project Number 2021YFA1001200, the start-up grant of CUHK, Early Career Scheme (No. 24301021) and General Research Fund (No. 14303022 \& 14301423) by Research Grants Council of Hong Kong from 2021-2023. The work of Z.Z. is supported by the National Key R\&D Program of China, Project Number 2021YFA1001200, and NSFC grant number 12031013, 12171013.  We thank Xu'an Dou for the help in numerical simulations, and Min Tang for the helpful discussions. 

\section*{Appendix} \label{sec:appendix}
We give a summary of the numerical discretization studied in \cite[Section 3]{jianguo2018}. A time-splitting method based on prediction-correction is proposed: 
\begin{equation}
\label{pc_system}
\begin{split}
\left\{\begin{array}{l} 
{ \partial_{t} \rho + \nabla\cdot(\rho \mathbf{u}) = \rho G(c) , } \\[4pt]
{ \partial_{t} u = m \nabla(\rho^{m-2}(\nabla\cdot (\rho\mathbf{u}) - \rho G(c))), }
\end{array} \quad 
\left\{\begin{array}{l}
\partial_t \rho=0 \\[4pt]
\partial_t \mathbf{u}=-\frac{1}{\varepsilon^2}\left(\mathbf{u}+\frac{m}{m-1} \nabla \rho^{m-1}\right) .
\end{array}\right.\right.
\end{split}
\end{equation}
Given $(\rho^n, \mathbf{u}^n)$, one solves the left system in \eqref{pc_system} for one time step and obtains the intermediate values 
$(\rho^{\ast}, \mathbf{u}^{\ast})$, then solve the second system in \eqref{pc_system} to get $(\rho^{n+1}, \mathbf{u}^{n+1})$. 

When $\varepsilon\to 0$, the second system in \eqref{pc_system} reduces to 
\begin{equation}
\partial_t \rho=0, \quad \mathbf{u}(x, t)=-\frac{m}{m-1} \nabla \rho^{m-1}(x, t). 
\end{equation}
In this projection step, notice that $\rho^{\ast} = \rho^{n+1}$. The time-splitting method for the fully relaxed system becomes 
\begin{equation}
\label{relax_system}
\begin{split}
\left\{\begin{array}{l}
\partial_t \rho+\nabla \cdot(\rho \mathbf{u})=\rho G(c), \\[4pt]
\partial_t \mathbf{u}=m \nabla\left(\rho^{m-2}(\nabla \cdot(\rho \mathbf{u})-\rho G(c))\right), 
\end{array}\right.
\quad \mathbf{u}(x, t)=-\frac{m}{m-1} \nabla \rho^{m-1}(x, t) .
\end{split}
\end{equation}
An implicit-explicit temporal discretization for the system \eqref{relax_system} is given as follows: 
\begin{equation}
\label{scheme_time}
\begin{aligned}
\frac{\mathbf{u}^{n*}-\mathbf{u}^n}{\Delta t} & =m \nabla\left(\left(\rho^n\right)^{m-2}\left(\nabla \cdot\left(\rho^n \mathbf{u}^{n *}\right)-\rho^n G\left(c^n, p\left(\rho^n\right)\right)\right)\right),  \\[4pt]
\frac{\rho^{n+1}-\rho^n}{\Delta t} & =-\nabla \cdot\left(\rho^n \mathbf{u}^{n *}\right)+\rho^{n+1} G\left(c^n, p\left(\rho_n\right)\right), \\[4pt]
\mathbf{u}^{n+1} & =-\frac{m}{m-1} \nabla\left(\rho^{n+1}\right)^{m-1}.
\end{aligned}
\end{equation}
Each of the equation above can be solved consecutively, which means that nonlinear solver is not needed in implementing the scheme. For the spatial discretization, we refer to \cite[Section 4]{jianguo2018} for details. 

In the 1D case, staggered grid for $u$ and regular grid for $\rho$ is used, namely
$$ \rho_i(t)=\frac{1}{\Delta x}\int_{x_{i-1/2}}^{x_{i+1/2}}\rho(x,t) dt, \qquad
u_{i+1/2}(t)=u(x_{i+1/2},t). $$
In the 1D case, the space discretization for $u^{n*}$ in \eqref{scheme_time} is by the centered finite difference method, 
\begin{equation}
\begin{aligned}
\frac{u_{i+1/2}^{n*} - u_{i+1/2}^n}{\Delta t} & = 
\frac{m}{\Delta x}\Big\{ (\rho_{i+1}^n)^{m-2}
\left( \frac{\rho_{i+3/2}^n u_{i+3/2}^{n *} - \rho_{i+1/2}^n u_{i+1/2}^{n *}}{\Delta x} - \rho_{i+1}^n G_i^n\right)  \\[4pt]
&\quad - (\rho_{i}^n)^{m-2}
\left( \frac{\rho_{i+1/2}^n u_{i+1/2}^{n *} - \rho_{i-1/2}^n u_{i-1/2}^{n *}}{\Delta x} - \rho_{i}^n G_i^n\right)
\Big\}, 
\end{aligned}
\end{equation}
where $G_i^n\approx G(x_i, n\Delta t)$ and the half grid values of $\rho$ are taken as 
$$ \rho_{i+1/2}^n = \frac{\rho_i^n + \rho_{i+1}^n}{2}. $$
In the second step of \eqref{scheme_time}, we use central scheme to discretize it. More specifically, 
$$ \frac{\rho_i^{n+1}-\rho_i^n}{\Delta t} + 
\frac{F_{i+1/2}^n - F_{i-1/2}^n}{\Delta x} = \rho_i^{n+1} G_i^n, $$
where the flux is given by
$$ F_{i\pm 1/2}^n = \frac{1}{2}\left[ \rho^{Ln} u^{n*} + \rho^{R n} u^{n*} - |u^{n*}| (\rho^{Rn} - \rho^{Ln})\right]_{i\pm 1/2}, $$
and $\rho_{i\pm 1/2}^{Ln}$ or $\rho_{i\pm 1/2}^{Rn}$
are edge values constructed as below. 
On the cell $[x_{i-1/2}, x_{i+1/2}]$, let 
$$ \rho_i^n(x) \approx \rho_i^n + (\partial_x\rho)_i^n (x-x_i). $$
At the interface $x_{i+1/2}$, the two approximations are given from the left or from the right, i.e., 
$$ \rho_{i+1/2}^{Ln} = \rho_i^n + \frac{\Delta x}{2}(\partial_x\rho)_i^n, \qquad \rho_{i+1/2}^{Rn} = \rho_{i+1}^n - \frac{\Delta x}{2}(\partial_x\rho)_{i+1}^n, $$
where $(\partial_x\rho)_i$ is computed by the minmod limiter \cite{CF12}. 
In the correction step of \eqref{scheme_time}, the centered difference approximation is employed, i.e., 
$$ u_{i+1/2}^{n+1} = -\frac{m}{m-1}\frac{(\rho_{i+1}^{n+1})^{m-1} - (\rho_{i}^{n+1})^{m-1}}{\Delta x}. $$

For the high-dimensional cases, the extension is straightforward and is thus omitted in this paper. Readers may refer to \cite{jianguo2018} for the explicit construction of the 2D schemes. 


\end{document}